\documentclass[11pt,english]{article}
\usepackage[T1]{fontenc}
\usepackage[utf8]{inputenc}
\usepackage{geometry}
\PassOptionsToPackage{obeyFinal}{todonotes}
\usepackage{color}
\usepackage{verbatim}
\usepackage{refstyle}
\usepackage{mathrsfs}
\usepackage{mathtools}
\usepackage{todonotes}
\usepackage{amsmath}
\usepackage{amsthm}
\usepackage{amssymb}
\usepackage{esint}
\usepackage{csquotes}
\usepackage[unicode=true,pdfusetitle,
 bookmarks=true,bookmarksnumbered=false,bookmarksopen=false,
 breaklinks=true,pdfborder={0 0 0},pdfborderstyle={},backref=false,colorlinks=true,citecolor=blue]
 {hyperref}

\makeatletter

\AtBeginDocument{\providecommand\eqref[1]{\ref{eq:#1}}}
\AtBeginDocument{\providecommand\thmref[1]{\ref{thm:#1}}}
\AtBeginDocument{\providecommand\remref[1]{\ref{rem:#1}}}
\AtBeginDocument{\providecommand\secref[1]{\ref{sec:#1}}}
\AtBeginDocument{\providecommand\corref[1]{\ref{cor:#1}}}
\AtBeginDocument{\providecommand\lemref[1]{\ref{lem:#1}}}
\AtBeginDocument{\providecommand\propref[1]{\ref{prop:#1}}}
\AtBeginDocument{\providecommand\subsecref[1]{\ref{subsec:#1}}}
\RS@ifundefined{subsecref}
  {\newref{subsec}{name = \RSsectxt}}
  {}
\RS@ifundefined{thmref}
  {\def\RSthmtxt{theorem~}\newref{thm}{name = \RSthmtxt}}
  {}
\RS@ifundefined{lemref}
  {\def\RSlemtxt{lemma~}\newref{lem}{name = \RSlemtxt}}
  {}

\numberwithin{equation}{section}
\numberwithin{figure}{section}
\numberwithin{table}{section}
\theoremstyle{plain}
\newtheorem{thm}{\protect\theoremname}[section]
\theoremstyle{remark}
\newtheorem{rem}[thm]{\protect\remarkname}
\theoremstyle{plain}
\newtheorem{lem}[thm]{\protect\lemmaname}
\theoremstyle{plain}
\newtheorem{cor}[thm]{\protect\corollaryname}
\theoremstyle{plain}
\newtheorem{prop}[thm]{\protect\propositionname}

\usepackage{refstyle}
\newref{rem}{name=Remark~,names=Remarks~}
\newref{lem}{name=Lemma~,names=Lemmas~}
\newref{def}{name=Definition~,names=Definitions~}
\newref{thm}{name=Theorem~,names=Theorems~}
\newref{prop}{name=Proposition~,names=Propositions~}
\newref{cond}{name=Condition~,names=Conditions~}
\newref{cor}{name=Corollary~,names=Corollaries~}
\newref{sec}{name=Section~,names=Sections~}
\newref{fig}{name=Figure~,names=Figures~}
\newref{subsec}{name=Subsection~,names=Subsections~}
\newref{appendix}{name=Appendix~,names=Appendices~}
\newref{eq}{name=,refcmd=(\ref{#1})}
\hypersetup{final}

\makeatother

\providecommand{\corollaryname}{Corollary}
\providecommand{\lemmaname}{Lemma}
\providecommand{\propositionname}{Proposition}
\providecommand{\remarkname}{Remark}
\providecommand{\theoremname}{Theorem}

\begin{document}
\global\long\def\e{\mathrm{e}}

\global\long\def\anon{\cdot}

\global\long\def\al{\alpha}

\global\long\def\R{\mathbf{R}}

\global\long\def\norm#1{\left\Vert #1\right\Vert }

\global\long\def\abs#1{\left|#1\right|}

\global\long\def\dist{\operatorname{dist}}

\global\long\def\diam{\operatorname{diam}}

\global\long\def\Hess{\operatorname{Hess}}

\global\long\def\Law{\operatorname{Law}}

\global\long\def\supp{\operatorname{supp}}

\global\long\def\spn{\operatorname{span}}

\global\long\def\tr{\operatorname{tr}}

\global\long\def\Re{\operatorname{Re}}

\global\long\def\Leb{\operatorname{Leb}}

\global\long\def\interior{\operatorname{interior}}

\global\long\def\Im{\operatorname{Im}}

\global\long\def\dif{\mathrm{d}}

\global\long\def\e{\mathrm{e}}

\global\long\def\p{\mathrm{p}}

\global\long\def\q{q}

\global\long\def\i{\mathrm{i}}

\global\long\def\Dif{\mathrm{D}}

\global\long\def\eps{\varepsilon}

\global\long\def\Cov{\operatorname{Cov}}

\global\long\def\Var{\operatorname{Var}}

\global\long\def\sgn{\operatorname{sgn}}

\global\long\def\width{\operatorname{width}}

\global\long\def\height{\operatorname{height}}

\global\long\def\AR{\operatorname{AR}}

\global\long\def\Euc{\mathrm{E}}

\global\long\def\Bernoulli{\operatorname{Bernoulli}}

\global\long\def\Uniform{\operatorname{Uniform}}

\global\long\def\Lip{\operatorname{Lip}}
\newcommand{\gu}{\textcolor{blue}}

 \title{The random heat equation in dimensions three and higher: the homogenization
viewpoint}
\author{Alexander Dunlap\thanks{Department of Mathematics, Stanford University, Stanford, CA 94305 USA. Current address: Courant Institute of Mathematical Sciences, New York University, New York, NY 10012 USA; \url{alexander.dunlap@cims.nyu.edu}.}\and Yu
Gu\thanks{Department of Mathematics, Carnegie Mellon University, Pittsburgh, PA 15213 USA; \url{yug2@andrew.cmu.edu}.}\and  Lenya
Ryzhik\thanks{Department of Mathematics, Stanford University, Stanford, CA 94305 USA; \url{ryzhik@stanford.edu}.}\and Ofer
Zeitouni\thanks{Department of Mathematics, Weizmann Institute of Science, POB 26, Rehovot 76100, Israel; \url{ofer.zeitouni@weizmann.ac.il}.}}
\maketitle

\begin{abstract}
We consider the stochastic heat equation 
$\partial_{s}u =\frac{1}{2}\Delta u +(\beta V(s,y)-\lambda)u$, with
a smooth space-time-stationary Gaussian random field $V(s,y)$,
in dimensions $d\geq 3$, with an initial condition~$u(0,x)=u_0(\eps x)$ 
and a suitably chosen $\lambda\in{\mathbb R}$. 
It is known that, for $\beta$ small enough,  the diffusively rescaled solution
$u^{\eps}(t,x)=u(\eps^{-2}t,\eps^{-1}x)$ converges weakly to
a scalar multiple of the solution~$\bar u(t,x)$ of the
heat equation with an effective diffusivity $a$, and that fluctuations converge, 
also in a weak sense, to the solution of the 
Edwards-Wilkinson equation with an effective noise strength~$\nu$ and the same
    effective diffusivity. 
    In this paper, we derive a pointwise approximation 
    $w^\eps(t,x)=\bar u(t,x)\Psi^\eps(t,x)+
    \eps u_1^\eps(t,x)$, where~$\Psi^\eps(t,x)=\Psi(t/\eps^2,x/\eps)$, 
    $\Psi$
    is a solution of the SHE with constant initial conditions, and  $u^\eps_1$ is
    an explicit corrector. We
    show that~$\Psi(t,x)$ converges to a stationary process
    $\tilde \Psi(t,x)$ as $t\to\infty$, that $\mathbf{E}|u^\eps(t,x)-w^\eps(t,x)|^2$ converges 
    pointwise  to~$0$ as~$\eps\to 0$, and that
    $\eps^{-d/2+1}(u^\eps-w^\eps)$ converges weakly to $0$ for fixed $t$.
    As a consequence, we derive new representations of the diffusivity $a$ and 
    effective noise strength~$\nu$. Our approach uses a
    Markov chain in the space of trajectories introduced in \cite{GRZ17}, as
  well as tools from homogenization theory. The corrector $u_1^\eps(t,x)$ is constructed using a seemingly new
  approximation scheme on a mesoscopic time scale.%
\end{abstract}

\section{Introduction}

We consider the long-time and large-space behavior of the solutions
$u(s,y)$ of the random heat equation with slowly varying initial
conditions
\begin{align}
\partial_{s}u & =\frac{1}{2}\Delta u+(\beta V(s,y)-\lambda)u,\label{eq:uPDE}\\
u(0,y) & =u_{0}(\eps y),\label{eq:uIC}
\end{align}
with $y\in\mathbb{R}^{d}$, $d\ge3$. Here, $u_{0}$ is a smooth, compactly-supported
initial condition, and the potential~$V(s,y)$ is a smooth, isotropic,
space-time-homogeneous, mean-zero Gaussian random field with a finite
correlation length. These assumptions are stronger than we truly need, but we make them to avoid distracting from the focus of the paper. 
We assume that $V(s,y)$ has the form
\[
V(s,y)=\int_{\mathbb{R}^{d+1}}\mu(s-s')\nu(y-y')\,\dif W(s',y'),
\]
where $\mu$ and $\nu$ are deterministic nonnegative functions of compact support,
such that $\nu$ is isotropic,
\[
\supp\mu\subset[0,1],\qquad\supp\nu\subset\{y\in\mathbb{R}^{d}\mid|y|\le1/2\},
\]
and $\dif W$ is a space-time white noise. From this, we see that
the covariance function is
\begin{equation}
R(s,y)\coloneqq\mathbf{E}V(s+s',y+y')V(s',y')=\int_{\mathbb{R}}\mu(s+t)\mu(t)\,\dif t\int_{\mathbb{R}^{d}}\nu(y+z)\nu(z)\,\dif z.\label{eq:Rdef}
\end{equation}
The constant $\lambda$ in \eqref{uPDE} will be chosen -- see \thmref{stationarity}
and \eqref{alphadef}--\eqref{alphaconverges} below -- so that $\mathbf{E}u(t,x)$
does not grow exponentially as $t\to\infty$. The small parameter
$\eps\ll1$ measures the ratio of the typical length scale of the
initial condition to the correlation length of the random potential.
As we are interested in the long-time behavior of $u$, we consider
its macroscopic rescaling
\[
u^{\eps}(t,x)=u(\eps^{-2}t,\eps^{-1}x),
\]
which satisfies the rescaled problem
\begin{align}
\partial_{t}u^{\eps} & =\frac{1}{2}\Delta u^{\eps}+\frac{1}{\eps^{2}}\left(\beta V(\eps^{-2}t,\eps^{-1}x)-\lambda\right)u^{\eps}\label{eq:uepsPDE}\\
u^{\eps}(0,x) & =u_{0}(x).\label{eq:uepsIC}
\end{align}
Here and throughout the paper, we use $s,y$ for the ``microscopic'' variables and $t=\eps^{2}s,x=\eps^{1}y$ for the rescaled ``macroscopic'' variables.
It was shown in \cite{GRZ17,magnen2017diffusive}, and also in \cite{mukherjee2017central}
at the level of the expectation, that there exists a $\beta_0>0$ so
that, if $0<\beta<\beta_{0}$, then there exists $\lambda$, depending
on $\beta,\mu,\nu$, and $a,\overline{c}>0$ so that, for any $t>0$,
\begin{equation}
v^{\eps}(t,\cdot)=\overline{c}u^{\eps}(t,\cdot)\label{eq:vepsdef}
\end{equation}
converges in probability and weakly in space as $\eps\to0$ to the solution $\overline{u}$ to
the homogenized problem
\begin{align}
\partial_{t}\overline{u} & =\frac{1}{2}a\Delta\overline{u}\label{eq:homogenizedPDE}\\
\overline{u}(0,x) & =u_{0}(x),\label{eq:homogenizedIC}
\end{align}
with an effective diffusivity $a\neq 1$.
It may come as a surprise that $\overline{c}\ne1$ in general; see \remref{cbar}
below. It was also shown that the fluctuations
\begin{equation}
\frac{1}{\eps^{d/2-1}}\left(v^{\eps}(t,\cdot)-\mathbf{E}v^{\eps}(t,\cdot)\right)\label{eq:fluctsconverge}
\end{equation}
converge in law and weakly in space as $\eps\to0$ to the solution
$\mathscr{U}$ of the Edwards--Wilkinson equation
\begin{align}
&\partial_{t}\mathscr{U}  =\frac{1}{2}a\Delta\mathscr{U}+\beta\nu\overline{u}\dif W\label{eq:EWPDE}\\
&\mathscr{U}(0,x)  \equiv0,\label{eq:EWIC}
\end{align}
with an effective noise strength $\nu>0$.

The results of \cite{GRZ17,magnen2017diffusive} concern weak convergence,
after  integration against a macroscopic test function. We note that the restriction to dimension~$d\geq 3$ is crucial: for $d= 2$ the behavior is different, as discussed in 
\cite{CSZ17} and \cite{CSZ18}.
In this work, we seek to understand the microscopic behavior of the
solutions, in the spirit of the classical random homogenization theory,
and explain how the microscopic behavior leads to the macroscopic
results of \cite{GRZ17,magnen2017diffusive}. We are also interested
in a more explicit interpretation of the macroscopic parameters: the
renormalization constant $\lambda$, the effective diffusivity $a$
in \eqref{homogenizedPDE}, the renormalization constant $\overline{c}$,
and the effective noise strength $\nu$ in~\eqref{EWPDE}. In particular,
we would like to connect these parameters to the classical objects
of stochastic homogenization.

As is standard in  PDE homogenization theory, we
introduce fast variables and consider a formal asymptotic expansion
for the solutions $u^{\eps}$ to \eqref{uepsPDE}--\eqref{uepsIC}
in the form
\begin{equation}
u^{\eps}(t,x)=u^{(0)}(t,x,\eps^{-2}t,\eps^{-1}x)+\eps u^{(1)}(t,x,\eps^{-2}t,\eps^{-1}x)+\eps^{2}u^{(2)}(t,x,\eps^{-2}t,\eps^{-1}x)+\cdots.\label{eq:expansion}
\end{equation}
Two issues commonly arise in such expansions. First, it may be hard
to prove, or even false, that the correctors
exist as stationary random fields. Second, the correlations of the higher-order correctors may
decay more slowly (in space) than those of lower-order correctors.
Thus, after integration against a test function, all terms in the
expansion may actually be of the same order, so including more correctors
may not improve the expansion from the perspective of the weak approximation.
We refer to \cite{gloriaotto,gumourrat} for a discussion of random
fluctuations in elliptic homogenization, and \cite{duerinckx2019higher,gu2017high} for a proof that stationary higher-order
correctors exist in sufficiently high dimensions.

In the present case, it is easy to see that the leading order term
in \eqref{expansion} should have the form
\begin{equation}
  u^{(0)}(t,x,\eps^{-2}t,\eps^{-1}x)=\overline{\overline{u}}(t,x)\Psi(\eps^{-2}t,\eps^{-1}x),\label{eq:leadingorder}
\end{equation}
where $\Psi$ is a solution to \eqref{uPDE} and 
does not depend on the initial condition $u_{0}$ in \eqref{uIC}, and $\overline{\overline{u}}$ is deterministic but depends on the initial condition $u_0$. We will see later that $\overline{\overline{u}}=\overline{u}$ with $\overline{u}$ taken to be the solution of the homogenized problem \eqref{homogenizedPDE}--\eqref{homogenizedIC}.
In the context of the usual homogenization theory, one would like
to think of $\Psi$ as being statistically stationary in space and
time. In the context of the Cauchy problem \eqref{uPDE}--\eqref{uIC}, it turns out that
better error bounds are achieved by letting $\Psi$ solve
the Cauchy problem with constant initial condition
\begin{equation}
\begin{aligned}
&\partial_{s}\Psi  =\frac{1}{2}\Delta\Psi+(\beta V-\lambda)\Psi\\
&\Psi(0,\cdot)  \equiv1.
\end{aligned}
\label{eq:Psiproblem}
\end{equation}
However, the intuition of a space-time-stationary $\Psi$ is still
justified, as we will see in \thmref{stationarity} below that~$\Psi$ in fact converges to a space-time-stationary
solution $\widetilde{\Psi}$ to \eqref{uPDE}.

As this paper was being written, we
learned of the very interesting recent paper \cite{CometsMukh} (see also the subsequent \cite{CCM19}), which considers (in our notation) the pointwise error $(\Psi(s,y)-\widetilde{\Psi}(s,y))/\Psi(s,y)$ in the case where the random potential $V$ is white in time, and shows that it is asymptotically Gaussian. This result is related but orthogonal to ours, and the proof techniques are quite different. 

\subsection*{Existence of a stationary solution and the leading-order term in
the expansion}

The renormalization constant $\lambda$ was understood in \cite{GRZ17}
as the unique value that keeps bounded the expectation of the solution
to \eqref{uPDE}. Our first result refines this explanation by showing
that, with this choice of $\lambda$, $\Psi(s,\cdot)$ in fact approaches
a space-time-stationary solution, which we call $\widetilde{\Psi}$, as $s\to\infty$. 
As remarked above, this shows that it is reasonable to take $\Psi^{\eps}(t,x)=\Psi(\eps^{-2}t,\eps^{-1}x)$
as a proxy for the stationary solution in the leading-order term for
the asymptotic expansion \eqref{expansion}. Note that neither~$\Psi$
nor its stationary limit $\widetilde{\Psi}$ depends on the initial condition $u_{0}$,
so both are ``universal'' objects.
\begin{thm}
\label{thm:stationarity}There is a $\beta_{0}>0$ so that for all
$0\le\beta<\beta_{0}$, there exists a $\lambda=\lambda(\beta)>0$ and
a space-time-stationary random function $\widetilde{\Psi}=\widetilde{\Psi}(s,y)>0$
that solves
\begin{equation}
\partial_{s}\widetilde{\Psi}(s,y)=\frac{1}{2}\Delta\widetilde{\Psi}(s,y)+(\beta V(s,y)-\lambda)\widetilde{\Psi}(s,y),\qquad s\in\mathbb{R},y\in\mathbb{R}^d,\label{eq:PsitildePDE}
\end{equation}
and there is a constant $C<\infty$ so that for any $y\in\mathbb{R}^{d}$ and $s>0$,
we have %
\begin{equation}
\mathbf{E}|\Psi(s,y)-\widetilde{\Psi}(s,y)|^{2}\le Cs^{-d/2+1}.\label{eq:Psiconvergence}
\end{equation}
\end{thm}

Throughout the paper, we will always assume that $\lambda=\lambda(\beta)$
is chosen as in the statement of \thmref{stationarity}. \thmref{stationarity}
can also be seen as an extension of \cite[Theorem 2.1]{MSZ16} to
the colored-noise setting, even though that result was formulated
in different terms. Some other relevant results in the literature
are \cite{DS80,TZ98}, which show the existence of stationary solutions
and convergence along subsequences in weighted $L^{2}$ spaces, also
in the white-noise setting.

The proof of \thmref{stationarity} is similar in spirit to that of
\cite[Theorem 2.1]{MSZ16}, but uses the framework of \cite{GRZ17}
to deal with the necessary renormalization parameter $\lambda$. For
the case of elliptic operators in divergence form, the existence of
stationary correctors in high dimensions was studied in \cite{armstrong,gloriaotto1,gloriaotto2},
and we refer the reader to the recent monograph \cite{armstrong1}
for a more complete list of references.

As an application of the existence of the stationary solution, we
will show in \secref{noisestrength} that the effective noise strength
$\nu$ in \eqref{EWPDE}, which has a complicated expression given
in \cite[(5.6)]{GRZ17}, has a more intuitive expression in terms
of the stationary solution. Let
\begin{equation}
G_{a}(t,x)=(2\pi at)^{-d/2}\exp\left\{ -|x|^{2}/(2at)\right\} \label{eq:Gadef}
\end{equation}
be the heat kernel with diffusivity $a$, and note that there exists
a constant $c$ so that
\begin{equation}
\int_{0}^{\infty}\int_{\mathbb{R}^{d}}G_{a}(r,z)G_{a}(r,z+x)\,\dif z\,\dif r=\frac{c}{a|x|^{d-2}}.\label{eq:cbardef}
\end{equation}

\begin{thm}
\label{thm:noiseexpr}For $0\le\beta<\beta_{0}$, with $\lambda$ taken as in \thmref{stationarity}, the effective noise
strength $\nu$ in \eqref{EWPDE} has the expression
\begin{equation}
\nu^{2}=\frac{a\lim\limits _{\eps\to0}\int\int g(x)g(\tilde{x})\eps^{-(d-2)}\Cov\left(\widetilde{\Psi}(0,\eps^{-1}x),\widetilde{\Psi}(0,\eps^{-1}\widetilde{x})\right)\,\dif x\,\dif\tilde{x}}{c\beta^{2}\e^{2\alpha_{\infty}}\int\int g(x)g(\tilde{x})|x-\tilde{x}|^{-(d-2)}\,\dif x\,\dif\tilde{x}}\label{eq:nu2expr}
\end{equation}
for any test function $g\in\mathcal{C}_{\mathrm{c}}^{\infty}(\mathbb{R}^{d})$.
The deterministic constant $\alpha_{\infty}$ is defined in \eqref{alphaconverges}
below.
\end{thm}

\thmref{noiseexpr} should be read as a weak formulation of the asymptotics
\[
\Cov(\widetilde{\Psi}(0,0),\widetilde{\Psi}(0,y))\sim\frac{c\beta^{2}\nu^{2}\e^{2\alpha_{\infty}}}{a|y|^{d-2}},\qquad|y|\gg1.
\]
In this sense, the effective noise strength in the Edwards--Wilkinson
equation \eqref{EWPDE} is directly related to the decay of the covariance
of the stationary solution. On the other hand, in \corref{Psicov}, we provide an expression for
the covariance term in \eqref{nu2expr} in terms of the Markov chain
introduced in~\cite{GRZ17} and reviewed in \secref{GRZreview} below.

Returning to the expansion \eqref{expansion}, the leading order term
in \eqref{leadingorder} is justified by the following microscopic
convergence result.
\begin{thm}
\label{thm:convergence}For $0\leq \beta <\beta_{0}$, with $\lambda$ taken as in \thmref{stationarity}, 
set $\Psi^{\eps}(t,x)=\Psi(\eps^{-2}t,\eps^{-1}x)$. If $u_{0}\in\mathcal{C}_{\mathrm{c}}^{\infty}(\mathbb{R}^{d})$,
then for all $t\ge 0$ and $x\in\R^d$ we have
\begin{equation}
\lim_{\eps\to0}\mathbf{E}|u^{\eps}(t,x)-\overline{u}(t,x)\Psi^{\eps}(t,x)|^{2}=0.\label{eq:convergence}
\end{equation}
\end{thm}

\begin{rem}
\label{rem:cbar}We can now explain the non-divergent renormalization
constant $\overline{c}$ in \eqref{vepsdef}. The function~$\Psi(s,\cdot)$
approaches a stationary solution~$\widetilde{\Psi}$
as~$s\to\infty$, that is, on a ``microscopically large'' time scale.   However, even
though $\Psi(0,\cdot)\equiv1$, it is not necessarily the case that
$\mathbf{E}\widetilde{\Psi}(s,\cdot)\equiv1$. (This \emph{would} be the case by the property of the It\^o integral if $V$ were white in time.) Thus we need to divide
by the factor of $\overline{c} =\mathbf{E}\widetilde{\Psi}(s,\cdot)$ to see convergence to the effective diffusion problem \eqref{homogenizedPDE}--\eqref{homogenizedIC}
with initial condition $u_{0}$ rather than $\overline{c}u_{0}$.%
\end{rem}

\subsection*{A higher-order approximation}

In order to obtain higher-order corrections in the asymptotic expansion,
if we plug \eqref{expansion} into \eqref{uepsPDE} and group terms
by powers of $\eps$, we obtain the following equations for $u_{1}$
and $u_{2}$:
\begin{equation}
\partial_{s}u_{1}(t,x,s,y)=\frac{1}{2}\Delta_{y}u_{1}(t,x,s,y)+(\beta V(s,y)-\lambda)u_{1}(t,x,s,y)+\nabla_{y}\Psi(s,y)\cdot\nabla_{x}\overline{u}(t,x),\label{eq:u1problem}
\end{equation}
and
\begin{equation}
\begin{aligned}\partial_{s}u_{2}(t,x,s,y) & =\frac{1}{2}\Delta_{y}u_{2}(t,x,s,y)+(\beta V(s,y)-\lambda)u_{2}(t,x,s,y)+\nabla_{y}\cdot\nabla_{x}u_{1}(t,x,s,y)\\
 & \qquad+\frac{1}{2}(1-a)\Psi(s,y)\Delta_{x}\overline{u}(t,x).
\end{aligned}
\label{eq:u2problem}
\end{equation}
As we will show in \secref{effective-diffusivity}, the effective
diffusivity $a$ can be recovered from a formal solvability condition
for \eqref{u2problem} to have a solution $u_{2}$ that is stationary
in the fast variables $s$ and $y$, which is a rather standard situation
in homogenization theory. However, here, as stationary correctors
are not expected to exist in low dimensions, justifying this expression
requires a construction of approximate correctors and passage to a
large-time limit, similar to the ``large box'' limit in elliptic
homogenization theory. In particular, \thmref{aSTthm} below shows how to evaluate the effective diffusivity in terms of objects familiar from the theory of homogenization.

Our last result concerns the connection between the local expansion
\eqref{expansion} and the weak approximation of the solution. As
we have mentioned, typically, the leading-order terms in such expansions
in stochastic homogenization only provide local approximations, while a 
control of the weak error (after integration against a test function)
requires extra terms. This is partly because the higher the order
of the corrector, the slower the decay of its covariance function,
leading to the accumulation of errors from terms of all orders. We
circumvent this issue in a way reminiscent of the ``straight-line''
approximation of trajectories on a mesoscopic time scale that is ``long but not too long'', such as
is used for models of particles in random velocity
fields or subject to random forces in \cite{KP79,KP80}.

If we look at \eqref{u1problem} for each macroscopic $t>0$ and $x\in\mathbb{R}^{d}$
fixed, as an evolution problem in $s$, we would have a ``complete
separation of scales'' factorization
\begin{equation}
u_{1}(t,x,s,y)=\sum_{k=1}^{d}\zeta^{(k)}(s,y)\frac{\partial\overline{u}(t,x)}{\partial x_{k}},\label{eq:separationofscales}
\end{equation}
where $\zeta^{(k)}$ solves the microscopic problem
\begin{equation}
\partial_{s}\zeta^{(k)}=\frac{1}{2}\Delta\zeta^{(k)}+(\beta V(s,y)-\lambda)\zeta^{(k)}+\frac{\partial\Psi}{\partial y_{k}}.\label{eq:zetakPDE}
\end{equation}
Instead of using \eqref{zetakPDE} directly, we consider {mesoscopic}
time intervals in $s$ of size $\eps^{-\gamma}$, with~$\gamma\in(0,2)$.
To be precise, for each $j\ge1$, let $\theta_{j}^{(k)}=\theta_{j}^{(k)}(s,y)$,
$1\le k\le d$, be the solution to
\begin{equation}
\begin{aligned}
&\partial_{s}\theta_{j}^{(k)}  =\frac{1}{2}\Delta_{y}\theta_{j}^{(k)}+(\beta V-\lambda)\theta_{j}^{(k)}+\frac{\partial\Psi}{\partial y_{k}},\qquad s>\eps^{-\gamma}(j-1),\\
&\theta_{j}^{(k)}(\eps^{-\gamma}(j-1),\cdot)  =0.
\end{aligned}
\label{eq:thetajproblem}
\end{equation}
Then, define $u_{1;j}=u_{1;j}(s,y)$ to be the solution to
\begin{equation}
\begin{aligned}
&\partial_{s}u_{1;j}  =\frac{1}{2}\Delta u_{1;j}+(\beta V-\lambda)u_{1;j},\qquad s>\eps^{-\gamma}j\\
&u_{1;j}(\eps^{-\gamma}j,y)  =\sum_{k=1}^{d}\theta_{j}^{(k)}(\eps^{-\gamma}j,y)\frac{\partial\overline{u}}{\partial x_{k}}(\eps^{2-\gamma}j,\eps y),
\end{aligned}
\label{eq:u1jproblem}
\end{equation}
and finally put
\begin{equation}
u_{1}^{\eps}(t,x)=\sum_{j=1}^{\lfloor\eps^{\gamma-2}t\rfloor}u_{1;j}(\eps^{-2}t,\eps^{-1}x)+\theta_{\lfloor\eps^{\gamma-2}t\rfloor+1}(\eps^{-2}t,\eps^{-1}x)\cdot\nabla\overline{u}(t,x).\label{eq:u1epsdef}
\end{equation}
This is similar to putting $s=\eps^{-2}t$, $y=\eps^{-1}x$ in the
formal PDE \eqref{u1problem}, except that rather than multiplying
the forcing by the ``current'' value of $\nabla\overline{u}$, we
multiply it by an ``out-of-date'' value of $\nabla\overline{u}$
that is only updated to the correct current value of $\nabla\overline{u}$
at times of the form $\eps^{-\gamma}j$, $j\in\mathbb{N}$. With this
definition of $u_{1}^{\eps}$, we have a weak convergence theorem
for the fluctuations. Recall that $\Psi^{\eps}(t,x)=\Psi(\eps^{-2}t,\eps^{-1}x)$ with $\Psi$ solving \eqref{Psiproblem}.
\begin{thm}
\label{thm:weakconvergence}Suppose that $0\le\beta<\beta_{0}$ and take $\lambda$ as in \thmref{stationarity}. Let $g\in\mathcal{C}_{\mathrm{c}}^{\infty}(\mathbb{R}^{d})$. Let $\gamma \in(0,2)$ and define $u_1^\eps$ as in \eqref{u1epsdef}. For
any $\zeta<(1-\gamma/2)\vee(\gamma-1)$ and any $t>0$, there exists
a $C>0$ (also depending on $\|u_0\|_{\mathcal{C}^3(\mathbb{R}^d)}$) so that
\begin{equation}
\mathbf{E}\left(\eps^{-d/2+1}\int g(x)[u^{\eps}(t,x)-\Psi^{\eps}(t,x)\overline{u}(t,x)-\eps u_{1}^{\eps}(t,x)]\,\dif x\right)^{2}\le C\eps^{2\zeta}.\label{eq:weakfluctsbound}
\end{equation}
\end{thm}

The optimal bound in \thmref{weakconvergence} is achieved when $\gamma=4/3$,
in which case $\zeta$ is required to be less than $1/3$.

We note that it would be hopeless to get a convergence-of-fluctuations
result like \thmref{weakconvergence}, even with an error of size
$\eps^{d/2-1}$ as in \eqref{fluctsconverge}, using only the first
term of the expansion \eqref{expansion} as in \thmref{convergence}.
This is because at that scale, \cite{GRZ17} gives different central
limit theorem statements for~$u$ and for $\Psi\overline{u}$: the
rescaled and renormalized fluctuations of $u$ converge to a solution
of the SPDE
\begin{equation}
\partial_{t}\mathscr{U}=\frac{1}{2}a\Delta\mathscr{U}+\beta\nu\overline{u}\dot{W},\label{eq:uSPDE}
\end{equation}
while the rescaled and renormalized fluctuations of $\Psi$ converge
to a solution of the SPDE
\begin{equation}
\partial_{t}\psi=\frac{1}{2}a\Delta\psi+\beta\nu\dot{W},\label{eq:psispde}
\end{equation}
and so the rescaled and renormalized fluctuations of $\Psi\overline{u}$
converge to a solution of the SPDE
\begin{align}
\partial_{t}(\psi\overline{u}) & =\frac{1}{2}a\overline{u}\Delta\psi+\beta\nu\overline{u}\dot{W}+\frac{1}{2}a\psi\Delta\overline{u}\nonumber \\
 & =\frac{1}{2}a\Delta(\psi\overline{u})-a\nabla\psi\cdot\nabla\overline{u}+\beta\nu\overline{u}\dot{W}.\label{eq:upsispde}
\end{align}
The limiting SPDEs \eqref{uSPDE} and \eqref{upsispde} are not the
same, so an extra correction, besides the first term $\Psi^\eps(t,x)\overline{u}$ of the homogenization expansion, is needed. This phenomenon is not new in the study of random fluctuations in homogenization, and has been discussed e.g. in \cite{gloriaotto,gumourrat}.

The definitions \eqref{thetajproblem}--\eqref{u1jproblem} sit midway
between two natural ways of interpreting the formal problem \eqref{u1problem}.
On one hand, \eqref{u1problem}, for fixed $x$ and $t$, can be solved
as in \eqref{separationofscales}--\eqref{zetakPDE}. However, defining
the corrector $u_{1}$ by \eqref{separationofscales}, with initial
condition $0$, and then evaluating at time $s=\eps^{-2}t$ does not
seem to yield a good convergence result, because $\nabla_{x}\overline{u}(\tau,x)$
is not constant on the time scale from~$\tau=0$ to $\tau=\eps^{2}s=t$.
On the other hand, \eqref{u1problem} could also be solved by 
plugging~$t=\eps^{2}s$,~$x=\eps y$ into \eqref{u1problem}, yielding the
PDE
\begin{equation}
\partial_{s}u_{1}(s,y)=\frac{1}{2}\Delta_{y}u_{1}(s,y)+(\beta V(s,y)-\lambda)u_{1}(s,y)+\nabla_{y}\Psi(s,y)\cdot\nabla_{x}\overline{u}(\eps^{2}s,\eps y).\label{eq:u1updateconstant}
\end{equation}
However, using a solution to \eqref{u1updateconstant} with initial condition $0$ also fails to yield a result along the lines of \thmref{weakconvergence}.
This is because the Feynman--Kac formula that arises from the
solution to~\eqref{u1updateconstant} involves the behavior of the
Markov chain of \cite{GRZ17} on microscopically short time scales,
while the limits appear to arise from the averaged behavior of the
Markov chain on long time scales. The delay in multiplying by $\nabla_{x}\overline{u}$
introduced by only updating its value at mesoscopic intervals allows
the short-time fluctuations to be averaged out, leaving only the averaged
behavior of the Markov chain, which allows us to deduce the limiting
behavior.

\subsection*{Proof strategies and the organization of the paper}

Although our study is in part motivated by the goal of understanding
results in the vein of \cite{GRZ17} from the perspective of PDE theory
and stochastic homogenization, our proofs remain probabilistic, relying
entirely on the Feynman--Kac formula. In particular, we extensively
use a certain Markov chain, introduced in \cite{GRZ17}, representing
the tilting of the measure on Brownian paths induced by the time-correlations
of the random potential $V$. Because this Markov chain is somewhat technical,
we start the paper by explaining how it appears via the Feynman-Kac
representation of the solution, and provide the definition
and properties of the Markov chain in \secref{GRZreview}, including
a few properties which were not needed in \cite{GRZ17}, and then
complete the proof of \thmref{stationarity} in Section~\ref{sec:stationarity}. 

The next two sections of the paper are devoted to the parameters $a$
and $\nu$ obtained in \cite{GRZ17}. In \secref{noisestrength},
we prove \thmref{noiseexpr} regarding the effective noise strength
$\nu$, showing that it is directly related to the spatial decay of
correlations of the stationary solution. In Section~\ref{sec:effective-diffusivity},
we first show how the effective diffusivity can be recovered from
the formal asymptotic expansion \eqref{expansion}; see expression
\eqref{oura} below. However, as the correctors that appear in the
asymptotic expansion may not be stationary (especially in lower dimensions),
this formula does not necessarily make sense directly. Instead, we
devise an approximation procedure via a sequence of problems on long
but finite time intervals and then pass to the limit. This is the
content of \thmref{aSTthm}. Finally, in the last two sections we
establish our convergence results for the formal asymptotic expansion:
the strong convergence (\thmref{convergence}) in \secref{convergence},
and the weak convergence (\thmref{weakconvergence}) in \secref{weakconvergence}.

As we have emphasized above, the Markov chain introduced in \cite{GRZ17}
plays a key technical role in the analysis throughout the paper. However,
the key observation of \cite{GRZ17} is that on microscopically
long time scales, the Markov chain mixes exponentially fast so that its partial sum essentially behaves like a Brownian motion. Our results still hold when the noise $V$ is taken to
be white rather than colored in time. In that case the Brownian motion is not tilted by the environment, and the Markov
chain is just its i.i.d.\ Gaussian increments, even on microscopically short
time scales. Thus, the reader may find it helpful on first reading
to ignore the time correlations and pretend that the Markov chain
is in fact a sequence of i.i.d.\ Gaussian random increments, which eliminates the need for most of
the technicalities introduced in \secref{GRZreview}. The analysis of \cite{GRZ17} constructing the Markov chain is orthogonal to the new applications of this chain in the present paper. %

\subsection*{Acknowledgments}

We would like to thank the anonymous referees for multiple helpful comments and suggestions which helped improve the presentation of the paper. AD was supported by an NSF Graduate Research Fellowship under grant
DGE-1147470, YG by NSF grant DMS-1613301/1807748/1907928 and the Center for Nonlinear Analysis at CMU, LR by NSF grant
DMS-1613603 and ONR grant N00014-17-1-2145, and OZ by an Israel Science
Foundation grant and funding from the European Research Council (ERC)
under the European Union's Horizon 2020 research and innovation program
(grant agreement number 692452). We would like to thank S.~Chatterjee,
F. Hernandez, G. Papanicolaou, and M. Perlman for helpful comments
and discussions.

\section{The tilted Brownian motion and the Markov chain\label{sec:GRZreview}}

All of the proofs in this paper rely heavily on a Markov chain introduced
in \cite{GRZ17} representing a tilted Wiener measure arising in the
Feynman--Kac representation of solutions to the stochastic heat equation.
In order to recall this Feynman--Kac representation, we first introduce
some notation. By $\mathbb{E}_{B}^{y}$ we denote expectation with
respect to the probability measure in which $B=(B^{1},\ldots,B^{d})$
is a standard $d$-dimensional Brownian motion with $B_{0}=y$, which we will always assume to be two-sided (i.e. running both forward and backward from time $0$) since this will be convenient in some formulas.
We use~$\mathbf{E}$ for expectation
with respect to the randomness in $V$, and use $\mathbb{E}$,
with various adornments, for expectation with respect to auxiliary
Brownian motions or Markov chains used in some way in the Feynman-Kac formula. Also, whenever we denote an expectation with a letter ``E,'' we will use the letter ``P'' with the same font and adornments to represent the corresponding probability measure. For any $s\in\mathbb{R}$ and $\mathfrak{A}\subset\mathbb{R}$,
we set
\begin{equation}
\mathscr{V}_{s;\mathfrak{A}}[B]=\int_{\mathfrak{A}}V(s-\tau,B_{\tau})\,\dif\tau.\label{eq:Vsadef}
\end{equation}
We will often use the shorthand $\mathscr{V}_{s}=\mathscr{V}_{s;[0,s]}$.
Thus, for example, the solution to \eqref{Psiproblem} can be expressed
in the Feynman--Kac representation
\begin{equation}
\Psi(s,y)=\mathbb{E}_{B}^{y}\exp\{\beta\mathscr{V}_{s}[B]-\lambda s\}.\label{eq:PsiFKpreliminary}
\end{equation}
There are, of course, also Feynman--Kac formulas for solutions to
the other equations in the introduction, which we will
write as they are needed.

\subsection{The tilted Brownian motion}

In computing moments of $\Psi(s,y)$, due to the Gaussianity of $\mathscr{V}_{\mathfrak{A}}[B]$,
it becomes necessary to evaluate the covariances of the latter. Recall the definition \eqref{Rdef} of the covariance kernel $R$ of the noise. We define,
for any pair of sets $\mathfrak{A},\widetilde{\mathfrak{A}}\subset\mathbb{R}$,
the quantity 
\begin{equation}
\mathscr{R}_{\mathfrak{A},\widetilde{\mathfrak{A}}}[B,\widetilde{B}]=\mathbf{E}\left(\mathscr{V}_{s;\mathfrak{A}}[B]\mathscr{V}_{s;\widetilde{\mathfrak{A}}}[\widetilde{B}]\right)=\int_{\widetilde{\mathfrak{A}}}\int_{\widetilde{\mathfrak{A}}}R(\tau-\tilde{\tau},B_{\tau}-\widetilde{B}_{\tilde{\tau}})\,\dif\tau\,\dif\tilde{\tau},\label{eq:scrRdef}
\end{equation}
which is independent of the choice of $s$ due to the stationarity of $V$, 
and use the
abbreviations 
\[
\mathscr{R}_{s,\tilde{s}}=\mathscr{R}_{[0,s],[0,\tilde{s}]},~~
\mathscr{R}_{\mathfrak{A}}=\mathscr{R}_{\mathfrak{A},\mathfrak{A}},
~~\mathscr{R}_{s}=\mathscr{R}_{s,s}.
\]
We will also abbreviate
$\mathscr{R}_{\bullet}[B]=\mathscr{R}_{\bullet}[B,B]$, where the
$\bullet$ can be replaced by any allowable subscript for $\mathscr{R}$, so that,
for example, 
\[
\mathscr{R}_{s,\tilde{s}}[B]=\mathscr{R}_{s,\tilde{s}}[B,B]
=\mathscr{R}_{[0,s],[0,\tilde{s}]}[B,B].
\]
As an example of the use of this notation, we have by Fubini's theorem
and the formula for the expectation of the integral of a Gaussian
that
\begin{equation}
\mathbf{E}\Psi(s,y)=\mathbb{E}_{B}^{y}\mathbf{E}\exp\{\beta\mathscr{V}_{s}[B]-\lambda s\}=\mathbb{E}_{B}^{y}\exp\left\{ \frac{\beta^{2}}{2}\mathscr{R}_{s}[B]-\lambda s\right\} .\label{eq:EPsi}
\end{equation}
Similarly, we can compute
\begin{align}
\mathbf{E}\Psi(s,y)\Psi(\tilde{s},\tilde{y}) & =\mathbb{E}_{B}^{y}\mathbb{E}_{\tilde{B}}^{\tilde{y}}\mathbf{E}\exp\left\{ \beta\mathscr{V}_{s}[B]+\beta\mathscr{V}_{\tilde{s}}^{\tilde{y}}[\widetilde{B}]-\lambda(s+\tilde{s})\right\} \nonumber \\
 & =\mathbb{E}_{B}^{y}\mathbb{E}_{\tilde{B}}^{\tilde{y}}\exp\left\{ \left(\frac{\beta^{2}}{2}\mathscr{R}_{s}[B]-\lambda s\right)+\beta^{2}\mathscr{R}_{s,\tilde{s}}[B,\widetilde{B}]+\left(\frac{\beta^{2}}{2}\mathscr{R}_{\tilde{s}}[\widetilde{B}]-\lambda\tilde{s}\right)\right\} .\label{eq:covariancecomputation}
\end{align}

We recognize the first and third terms in the last exponential from
the exponential in \eqref{EPsi}. This motivates the definition of
the tilted path measure $\widehat{\mathbb{P}}_{B;\bullet}^{y}$ by
\begin{equation}
\widehat{\mathbb{E}}_{B;\bullet}^{y}\mathscr{F}[B]=\frac{1}{Z_{\bullet}}\mathbb{E}_{B}^{y}\left[\mathscr{F}[B]\exp\left\{ \frac{1}{2}\beta^{2}\mathscr{R}_{\bullet}[B]\right\} \right],\qquad\qquad Z_{\bullet}=\mathbb{E}_{B}^{y}\exp\left\{ \frac{1}{2}\beta^{2}\mathscr{R}_{\bullet}[B]\right\} \label{eq:tilting}
\end{equation}
for any measurable functional $\mathscr{F}$ on the space $\mathcal{C}([0,\infty);\mathbb{R}^{d})$,
where $\bullet$ can be taken to be any of the allowable subscripts
for $\mathscr{R}$. We also define $\widehat{\mathbb{P}}_{B,\widetilde{B};\bullet}^{y,\tilde{y}}=\widehat{\mathbb{P}}_{B;\bullet}^{y}\otimes\widehat{\mathbb{P}}_{\widetilde{B};\bullet}^{\tilde{y}}$
and denote by $\widehat{\mathbb{E}}_{B,\widetilde{B};\bullet}^{y,\tilde{y}}$
the corresponding expectation. Finally, we define
\begin{equation}
\alpha_{s}=\log Z_{s}-\lambda s,\label{eq:alphadef}
\end{equation}
and note that, according to \cite[Lemma A.1]{GRZ17} and its proof,
there exists a unique $\lambda=\lambda(\beta)$ so that
\begin{equation}
|\alpha_{s}-\alpha_{\infty}|\le C\e^{-cs}\label{eq:alphaconverges}
\end{equation}
for some $\alpha_{\infty}\in(0,\infty)$, $c>0$, and $C<\infty$.
This is where the constant $\lambda$ comes from, and we fix it for
the rest of the paper. This definition of $\lambda$ should be interpreted
in terms of \eqref{EPsi}: $\lambda$ is chosen so that $\mathbf{E}\Psi(s,y)$
remains of order $O(1)$ as $s\to+\infty$. Equivalently, it is the exponential
rate of growth of the unrenormalized, that is, with $\lambda=0$, 
multiplicative stochastic
heat equation with noise strength $\beta$. We note that a  consequence of \thmref{stationarity} is that $\e^{\alpha_\infty}=\mathbf{E}\widetilde{\Psi}(s,y)=\overline{c}$, where $\overline{c}$ is as in \eqref{vepsdef}. Another consequence is that
\begin{equation}
 \lambda = \frac{\beta \mathbf{E} \widetilde{\Psi}(t,x)V(t,x)}{\mathbf{E} \widetilde{\Psi}(t,x)}.\label{eq:lambda}
\end{equation}
This allows $\lambda$ to be recovered directly from the law of the stationary solution. The problem \eqref{PsitildePDE} already depends on $\lambda$, so we cannot use this expression as a definition of $\lambda$. However, if as in \eqref{Psiconvergence} we approximate $\tilde{\Psi}$ by $\Psi$ evaluated at a large time, then the right side of the resulting version of \eqref{lambda} does not depend on the choice of $\lambda$ in \eqref{uPDE}. Thus we could \emph{define}

\begin{equation}
 \lambda = \frac{\lim\limits_{t\to\infty}\beta \mathbf{E} \Psi(t,x)V(t,x)}{\lim\limits_{t\to\infty}\mathbf{E} \Psi(t,x)}.\label{eq:lambda2}
\end{equation}

An example of the utility of this tilted measure is that it lets us
rewrite \eqref{covariancecomputation} by
\begin{equation}
\mathbf{E}\Psi(s,y)\Psi(\tilde{s},\tilde{y})=\e^{\alpha_{s}+\alpha_{\tilde{s}}}\widehat{\mathbb{E}}_{B,\widetilde{B};s,\tilde{s}}^{y,\tilde{y}}\exp\left\{ \beta^{2}\mathscr{R}_{s,\tilde{s}}[B,\widetilde{B}]\right\} .\label{eq:Psicov}
\end{equation}
In light of \eqref{alphaconverges}, the factor $\e^{\alpha_{s}+\alpha_{\tilde{s}}}$
should be thought of, for (``microscopically'') large $s,\tilde{s}$,
as essentially a constant. This expression is
analogous to the computation in \cite[Lemma 3.1]{MSZ16}, to which
it indeed reduces if $V$ is taken to be white in time rather than
colored as it is in our setting. Indeed, in the white-in-time case,
the kernel $R(s,y)$ becomes a delta mass in $s$ at $s=0$, and thus
the quantity $\mathscr{R}_{s}[B]$ becomes the constant $\lambda s$
(the Itô--Stratonovich correction), not depending on $B$, so also
$\alpha_{\infty}=0$. In particular, in the white-in-time case
the tilting \eqref{tilting} becomes trivial: we use the tilting 
to account for the time-correlations of the noise. Then from \eqref{Psicov}
we recover exactly the first display in the proof of \cite[Lemma 3.1]{MSZ16}.

\subsection{The Markov chain}

A key point of \cite{GRZ17} is that a Brownian motion tilted according
to \eqref{tilting} can be approximately represented by a Markov chain.
Since $R(s,y)$ is supported on times $s\in[-1,1]$, the 
functional~$\mathscr{R}_{\bullet}[B]$ that appears in \eqref{tilting} only
involves interactions between the values of $B$ at times of distance
at most~$2$ from each other. Thus, if we ``chunk'' the Brownian
motion into segments of length $1$, the tilting only takes into account
the interactions between each segment and the immediate preceding
and succeeding segments. One can then represent the tilted Brownian
motion as a Markov chain on the chunks, with the caveat that another, ultimately small,
tilting is needed to account for the edge effects at time $T$.

It
is shown in \cite{GRZ17} that the Markov chain satisfies the \emph{Doeblin
condition}, which is to say that the transition measures uniformly
majorize a (small) multiple of the stationary measure. This condition
is an elementary tool in the theory of Markov chains; see e.g. \cite{MT09}
for an introduction. Therefore, at every step of the chain
corresponding to a length-$1$ chunk of the Brownian motion, there
is a probability bounded away from zero that the next step of the
chain can be considered to be sampled from the stationary distribution.
Conditional on this event occurring at a particular step,
the chain is then at its stationary distribution.
Therefore, the chain converges to its stationary distribution exponentially quickly. %

We state these ideas precisely in the following theorem, which summarizes
several results and discussions in \cite{GRZ17}. We let $\Xi_{T}=\{\omega\in\mathcal{C}([0,T])\mid\omega(0)=0\}$,
and, given $W_{i}\in\Xi_{T_{i}}$, we define $[W_{1},\ldots,W_{k}]\in\Xi_{\sum_{i}T_{i}}$
by concatenating the increments, as in \cite[(4.2)]{GRZ17}.
\begin{thm}[\cite{GRZ17}]
\label{thm:BstoWs}Let $T>1$ and $N=\lfloor T\rfloor-1$. There
is a Markov chain $w_{0},w_{1},\ldots,w_{N},w_{N+1}$, with $w_{0}\in\Xi_{T-[T]}$
and $w_{j}\in\Xi_{1}$ for $1\le j\le N+1$, which has the following
properties.
\begin{enumerate}
\item (Time-homogeneity.) The transition probability measure
\[
\widehat{\pi}(w_{j},\cdot)=\Law(w_{j+1}\mid w_{j})
\]
does not depend on $j$ for $j=1,\ldots,N-1$.
\item (Relationship with the tilted Brownian motion.) There is a bounded, measurable, even functional $\mathscr{G}:\Xi_1\to\mathbf{R}$ such that,if we put $W=[w_{0},\ldots,w_{N+1}]\in\Xi_{T}$, and let $\widetilde{\mathbb{E}}_W$ denote expectation with respect to the measure in which $W$ is obtained from the Markov chain,
then we have, for any bounded continuous function $\mathscr{F}$ on
$\Xi_{T}$, that
\begin{equation}
\widehat{\mathbb{E}}_{B;T}\mathscr{F}[B]=\widetilde{\mathbb{E}}_{W}[\mathscr{F}[W]\mathscr{G}[w_{N}]].\label{eq:Markovchainexpectation}
\end{equation}
\item (Doeblin condition.) There is a sequence of i.i.d. Bernoulli random
variables $\eta_{j}^{W}$, $j=1,2,\ldots$, with success probability
not depending on $T$, so that
\[
\Law(w_{j}\mid\eta_{j}^{W}=1,\{w_{i}\;:\;i<j\})=\overline{\pi},
\]
where $\overline{\pi}$ is the invariant measure of $\widehat{\pi}$.
\end{enumerate}
\end{thm}

\thmref{BstoWs} summarizes several results of \cite{GRZ17}. The
Markov chain $(w_{k})$ is constructed in \cite[Section 4.1]{GRZ17}.
Equation \eqref{Markovchainexpectation} is \cite[(4.25)]{GRZ17},
where we use the notation $\mathscr{G}$ instead of $\mathscr{G}_{\eps}$
because the functional in fact does not depend on the $\eps$ of \cite{GRZ17}
(which is the same as the $\eps$ in the present paper, but is playing
no role in the present discussion). The functional $\mathscr{G}$
represents the additional tilting to account for edge effects at time
$T$. This additional tilting should be thought of as an error term
and in our arguments we will always strive to show that it does not
play an important role; the reader who pretends that $\mathscr{G}\equiv1$
will not miss the thrust of the arguments in the paper. The Doeblin
condition is established in \cite[(4.18)]{GRZ17}, as explained in
the discussion surrounding \cite[(4.27)]{GRZ17}.

We note again that \thmref{BstoWs} is trivial in the case
when $V$ is white in time: then the Markov chain is simply given
by the independent increments of the Brownian motion, and is always
at its stationary distribution.

We will use the notation
\begin{equation}
\widetilde{\mathbb{E}}_{W}^{y}\mathscr{F}[W]=\widetilde{\mathbb{E}}_{W}\mathscr{F}[y+W].\label{eq:Etildey}
\end{equation}
Define the stopping times $\sigma_{0}^{W}=0$, $\sigma_{n}^{W}=\min\{t\ge\sigma_{n-1}^{W}\mid\eta_{t}^{W}=1\}$
and put, for $n\ge0$,
\begin{equation}
\mathbf{W}_{n}^{W}=W_{\sigma_{n+1}^{W}}-W_{\sigma_{n}^{W}}.\label{eq:independentchunks}
\end{equation}
This is the construction in \cite[(4.27)]{GRZ17}. The following lemma
summarizes some results of \cite{GRZ17} about these stopping times.
\begin{lem}
\label{lem:Ws}The family $\{\mathbf{W}_{n}^{W}\}_{n\ge0}$ is a collection
of independent, statistically isotropic random variables with exponential
tails. Moreover, the elements of $\{\mathbf{W}_{n}^{W}\}_{n\ge1}$
are identically distributed.
\end{lem}

\begin{proof}
The fact that $\{\mathbf{W}_{n}^{W}\}_{n\ge0}$ is independent, and
that the elements of $\{\mathbf{W}_{n}^{W}\}_{n\ge1}$ are identically
distributed, is an immediate consequence of the Doeblin condition
and the time-homogeneity of the Markov chain. Isotropy follows from
the isotropy of the construction. Exponential tails were established
in \cite[Lemma A.2]{GRZ17}.
\end{proof}
The construction leading to \eqref{independentchunks} can be applied
to pairs of paths as well, as explained at the end of \cite[Section 4.1]{GRZ17}.
Given two independent copies $W,\widetilde{W}$ of the Markov chain,
define
\[
\eta_{j}^{W,\widetilde{W}}=\eta_{j}^{W}\eta_{j}^{\widetilde{W}},
\]
and the stopping times
\begin{equation}
\sigma_{n}^{W,\widetilde{W}}=\begin{cases}
0 & n=0;\\
\min\{t\ge\sigma_{n-1}\;:\;\eta_{t}^{W,\widetilde{W}}=1\} & n\ge1.
\end{cases}\label{eq:sigmanWWdef}
\end{equation}
Then put
\[
\mathbf{W}_{n}^{W,\widetilde{W}}=W_{\sigma_{n+1}^{W,\widetilde{W}}}-W_{\sigma_{n}^{W,\widetilde{W}}},\qquad\qquad\widetilde{\mathbf{W}}_{n}^{W,\widetilde{W}}=\widetilde{W}_{\sigma_{n+1}^{W,\widetilde{W}}}-\widetilde{W}_{\sigma_{n}^{W,\widetilde{W}}}.
\]
Analogously to \eqref{Etildey}, we use the notation $\widetilde{\mathbb{P}}_{W,\widetilde{W}}^{y,\tilde{y}}=\widetilde{\mathbb{P}}_{W}^{y}\otimes\widetilde{\mathbb{P}}_{\widetilde{W}}^{\tilde{y}}$.
We have the following corollary of \lemref{Ws}.
\begin{cor}\label{cor:chunkprops}
The family $\{\mathbf{W}_{n}^{W,\widetilde{W}}\}_{n\ge0}\cup\{\widetilde{\mathbf{W}}_{n}^{W,\widetilde{W}}\}_{n\ge0}$
is a collection of independent isotropic random variables with exponential
tails.\footnote{We use the standard terminology that a random variable $X$ has exponential tails if there are constants $C,c>0$ such that $\mathbf{P}(|X|>x)\le C\e^{-cx}$ for all $x>0$.} Moreover, the elements of $\{\mathbf{W}_{n}^{W,\widetilde{W}}\}_{n\ge1}\cup\{\widetilde{\mathbf{W}}_{n}^{W,\widetilde{W}}\}_{n\ge1}$
are identically distributed.
\end{cor}

Now let us set
\begin{equation}
\kappa_{1}=\mathbb{P}(\eta_{j}^{W}=1),\qquad\qquad\kappa_{2}=\mathbb{P}(\eta_{j}^{W,\widetilde{W}}=1)=\kappa_{1}^{2}.\label{eq:kappadef}
\end{equation}
 The next proposition gives an expression for the effective diffusivity
$a$ in \eqref{homogenizedPDE} in terms of the Markov chain.
\begin{prop}[{\cite[Proposition 4.1]{GRZ17}}]
\label{prop:invariance}There is a diagonal $d\times d$ matrix
\begin{equation}
\mathbf{a}=aI_{d\times d}=\kappa_{1}\widetilde{\mathbb{E}}_{W}[\mathbf{W}_{n}^{w}(\mathbf{W}_{n}^{W})^{t}]\label{eq:adef}
\end{equation}
so that for any $t>0$, as $\eps\to0$, the process $\{\eps W_{\eps^{2}\tau}\}_{0\le\tau\le t}$ (under the measure $\widetilde{\mathbb{P}}_W$)
converges in distribution in $\mathcal{C}([0,t])$ to a Brownian motion
with covariance matrix $\mathbf{a}$.
\end{prop}

Two Brownian motions in $d\ge3$ will almost surely spend at most
a finite amount of time within distance $1$ of each other. The fact
that this is also true for the Markov chains $W,\widetilde{W}$ is
expressed in the next two propositions, and will play a crucial role
in the sequel.
\begin{prop}[{\cite[Corollary 4.4]{GRZ17}}]
\label{prop:expexpbdd}There is a $\beta_{0}>0$ and a deterministic
constant $C<\infty$ so that if $0\le\beta<\beta_{0}$ then for any
$s\ge0$, $y,\tilde{y}\in\mathbb{R}^{d}$, we have
\[
\widetilde{\mathbb{E}}_{W,\widetilde{W}}^{y,\tilde{y}}\left[\exp\left\{ \beta^{2}\mathscr{R}_{[s,\infty)}[W,\widetilde{W}]\right\} \;\middle|\;\mathcal{F}_{s}\right]\le C
\]
with probability $1$, where $\mathcal{F}_{s}$ is the $\sigma$-algebra
generated by the paths $W,\widetilde{W}$ on the time interval $[0,s]$.
\end{prop}

We will require a slightly stronger version of \propref{expexpbdd},
which can be proved similarly.
\begin{prop}
\label{prop:expexpbdd-bridge}There is a $\beta_{0}>0$ and a deterministic
constant $C<\infty$ so that if $0\le\beta<\beta_{0}$ then for all
$r,\tilde{r}>0$, we have
\[
\widetilde{\mathbb{E}}_{W,\widetilde{W}}^{y,\tilde{y}}\left[\exp\left\{ \beta^{2}\mathscr{R}_{\infty}[W,\widetilde{W}]\right\} \;\middle|\;\mathcal{F}_{r,\tilde{r}}\right]\le C
\]
with probability $1$, where $\mathcal{F}_{r,\tilde{r}}$ is the $\sigma$-algebra
generated by the path $W$ on $[0,r]$ and the path $\widetilde{W}$
on $[0,\tilde{r}]$.
\end{prop}

We also need some estimates from \cite{GRZ17} on various error terms.
\begin{lem}[{\cite[(4.30)]{GRZ17}}]
There is a constant $C$ so that
\begin{equation}
\widetilde{\mathbb{E}}_{W}|\eps W_{\eps^{-2}t_{2}}-\eps W_{\eps^{-2}t_{1}}|^{2}\le C(t_{2}-t_{1}).\label{eq:West}
\end{equation}
\end{lem}

\begin{lem}[{\cite[Lemma A.3]{GRZ17}}]
\label{lem:cutofftail}For any $\chi>0$, there are constants $0<c,C<\infty$
so that if, for each $T$, $\mathscr{F}_{T}:\Xi_{T}\to\mathbf{R}$
is a bounded functional on $\Xi_{T}$, and $\{S_{n}\},\{T_{n}\}$
are sequences of real numbers such that $S_{n},T_{n},S_{n}-T_{n}\to+\infty$,
then
\begin{multline*}
\left|\widetilde{\mathbb{E}}_{W}\mathscr{F}_{T_{n}}[W|_{[0,T_{n}]}]-\widetilde{\mathbb{E}}_{W}\mathscr{F}_{T_{n}}[W|_{[0,T_{n}]}]\mathscr{G}(w_{S_{n}})\right|\\
\le C\left(\widetilde{\mathbb{E}}_{W}\left(\mathscr{F}_{T_{n}}[W|_{[0,T_{n}]}]\right)^{\chi}\right)^{1/\chi}\exp\left\{ -c(T_{n}\wedge(S_{n}-T_{n}))\right\} .
\end{multline*}
\end{lem}

Here, $\mathscr{G}$ is as in \thmref{BstoWs}. The rate of convergence
is not stated explicitly in \cite[Lemma A.3]{GRZ17}, but it comes
from the proof there.
\begin{lem}[{\cite[Lemma A.2]{GRZ17}}]
\label{lem:exptails}We have constants $0<c,C<\infty$ so that
\[
\widetilde{\mathbb{P}}_{W,\widetilde{W}}^{x,\tilde{x}}\left[\max_{r,\tilde{r}\in[\sigma_{n},\sigma_{n+2}]}\left(\left|W_{r}-W_{\sigma_{n}^{W,\widetilde{W}}}\right|+\left|\widetilde{W}_{\tilde{r}}-\widetilde{W}_{\sigma_{n}^{W,\widetilde{W}}}\right|\right)>a\right]\le C\e^{-ca}.
\]
\end{lem}

\subsection{Estimates on path intersections\label{subsec:intersections}}

These preliminaries having been completed, we now prove a fact that
will be essential for us: %
that two independent copies of the Markov chain, started at distance of order
$\eps^{-1}$ from each other, pass within distance $1$ of each other with probability $\eps^{d-2}$.
This is the same situation as for the standard Brownian motion. Explicitly, we prove the following (which does not in fact require the assumption that $\beta$ is small).
\begin{prop}
\label{prop:hittingprob}There is a constant $C$ so that
\[
\widetilde{\mathbb{P}}_{W,\widetilde{W}}^{x,\tilde{x}}\left[\inf_{\substack{r,\tilde{r}>0\\
|r-\tilde{r}|\le1
}
}|W_{r}-\widetilde{W}_{\tilde{r}}|\le1\right]\le\frac{C}{|x-\tilde{x}|^{d-2}}.
\]
\end{prop}

In order to prove \propref{hittingprob}, we first prove it just at regeneration times. For the rest of this section, to
economize on notation we put $\sigma_{n}\coloneqq\sigma_{n}^{W,\widetilde{W}}$ (defined in \eqref{sigmanWWdef}).
\begin{lem}
\label{lem:endpointsok}For all $A>0$, we have
\[
\widetilde{\mathbb{P}}_{W,\widetilde{W}}^{x,\tilde{x}}\left[\inf_{n\ge0}\left|W_{\sigma_{n}}-\widetilde{W}_{\sigma_{n}}\right|\le A\right]\le\frac{A^{d-2}}{|x-\tilde{x}|^{d-2}}.
\]
\end{lem}

\begin{proof}
Let
\[
X_{n}=W_{\sigma_{n}}-\widetilde{W}_{\sigma_{n}},
\]
let $\mathcal{H}_{n}$ be the $\sigma$-algebra generated by $X_{1},\ldots,X_{n}$,
and set
\[
q(z)=\frac{1}{(|z|\vee A)^{d-2}}.
\]
For any $z\in\mathbb{R}^{d}$ with $|z|\ge A$ and $M>0$, if we let
$\dif S$ denote the surface measure on $\{|\tilde{z}-z|=M\}$, then
we have
\begin{equation}
\fint_{|\tilde{z}-z|=M}q(\tilde{z})\,\dif S(\tilde{z})\le\fint_{|\tilde{z}-z|=M}\frac{1}{|\tilde{z}|^{d-2}}\,\dif S(\tilde{z})\le\frac{1}{|z|^{d-2}}=q(z)\label{eq:qsuperharmonic}
\end{equation}
by the mean value inequality for superharmonic functions, as $z\mapsto|z|^{-d+2}$
is superharmonic. Here, the notation $\fint$ means that we normalize
the surface measure to have total mass $1$. Let $\omega$ be the
smallest $n$ so that $|X_{n}|\le A$, or $\infty$ if $|X_{n}|>A$
for all $n$. Note that $\omega$ is a stopping time with respect to the filtration
$\{\mathcal{H}_{n}\}$. Also, the distribution of $X_{n}-X_{n-1}$
is isotropic and independent of $\mathcal{H}_{n-1}$ for each $n\ge1$ by \corref{chunkprops}. Therefore, we have, whenever $n-1<\omega$,
\begin{align*}
\mathbb{\widetilde{E}}_{W,\widetilde{W}}[q(X_{n})\mid\mathcal{H}_{n-1}] & =\int_{\mathbb{R}^{d}}q({z})\,\dif\mathbb{\widetilde{P}}_{W,\widetilde{W}}(X_{n}={z}\mid\mathcal{H}_{n-1})\\
 & =\int_{\mathbb{R}}\fint_{|z-X_{n-1}|=M}q({z})\,\dif S({z})\,\dif\mathbb{\widetilde{P}}_{W,\widetilde{W}}(|X_{n}-X_{n-1}|=M)\\
 & \le\int_{\mathbb{R}}q(X_{n-1})\,\dif\mathbb{\widetilde{P}}_{W,\widetilde{W}}(|X_{n}-X_{n-1}|=M)=q(X_{n-1}),
\end{align*}
where the last inequality is by \eqref{qsuperharmonic}. Thus, the sequence $(q(X_{n\wedge\omega}))_{n}$
is a supermartingale. By the optional stopping theorem, for any $n$
we have
\[
\frac{1}{|x-\tilde{x}|^{d-2}}=q(X_{0})\ge\widetilde{\mathbb{E}}_{W,\widetilde{W}}q(X_{n\wedge\omega})\ge\frac{1}{A^{d-2}}\widetilde{\mathbb{P}}_{W,\widetilde{W}}(\omega\le n).
\]
Therefore, we have
\[
\widetilde{\mathbb{P}}_{W,\widetilde{W}}(\omega<\infty)\le\frac{A^{d-2}}{|x-\tilde{x}|^{d-2}}
\]
by Fatou's lemma.
\end{proof}
\begin{proof}[Proof of \propref{hittingprob}.]
Let
\[
B_{n}=\max_{r,\tilde{r}\in[\sigma_{n},\sigma_{n+2}]}\left(\left|W_{r}-W_{\sigma_{n}^{W,\widetilde{W}}}\right|+\left|\widetilde{W}_{\tilde{r}}-\widetilde{W}_{\sigma_{n}^{W,\widetilde{W}}}\right|\right)
\]
and
\[
\omega_{M}=\inf\left\{ n\ge0\;:\;|W_{\sigma_{n}}-\widetilde{W}_{\sigma_{n}}|\le2^{M}\right\} .
\]
We have
\begin{align}
 & \left\{ \inf_{|r-\tilde{r}|\le1}|W_{r}-\widetilde{W}_{\tilde{r}}|\le1\right\} \subseteq\bigcup_{M=0}^{\infty}\bigcup_{n=0}^{\infty}\left(\left\{ |W_{\sigma_{n}}-\widetilde{W}_{\sigma_{n}}|\le2^{M}\right\} \cap\left\{ B_{n}\ge2^{M-1}-1\right\} \right)\nonumber \\
 & \qquad\subseteq\bigcup_{M=0}^{\infty}\left[\{\omega_{M}<\infty\}\cap\left(\bigcup_{n=\omega_{M}}^{\infty}\left(\left\{ |W_{\sigma_{n}}-\widetilde{W}_{\sigma_{n}}|\le2^{M}\cap\left\{ B_{n}\ge2^{M-1}-1\right\} \right\} \right)\right)\right].\label{eq:breakup}
\end{align}
Therefore, we can estimate, abbreviating $\mathbb{P}=\widetilde{\mathbb{P}}_{W,\widetilde{W}}^{x,\tilde{x}}$
and letting the constant $C$ change from line to line,
\begin{align*}
\mathbb{P} & \left[\inf_{|r-\tilde{r}|\le1}|W_{r}-\widetilde{W}_{\tilde{r}}|\le1\right]\\
 & \le\sum_{M,\ell=0}^{\infty}\mathbb{P}(\omega_{M}=\ell)\sum_{n=\ell}^{\infty}\mathbb{P}\left[|W_{\sigma_{n}}-\widetilde{W}_{\sigma_{n}}|\le2^{M}\;\middle|\;\omega_{M}=\ell\right]\mathbb{P}\left[B_{n}\ge2^{M-1}-1\right]\\
 & \le C\sum_{M=0}^{\infty}\e^{-c(2^{M-1}-1)}\sum_{\ell=0}^{\infty}\mathbb{P}(\omega_{M}=\ell)\sum_{n=\ell}^{\infty}\frac{2^{Md}}{(n-\ell+1)^{d/2}}=C\sum_{M=0}^{\infty}\e^{-c(2^{M-1}-1)+CMd}\mathbb{P}(\omega_{M}<\infty)\\
 & \le C\sum_{M=0}^{\infty}\e^{-c(2^{M-1}-1)+CMd}\cdot\frac{2^{(d-2)M}}{|x-\tilde{x}|^{d-2}}\le\frac{C}{|x-\tilde{x}|^{d-2}},
\end{align*}
where the first inequality is by \eqref{breakup}, the second is by
\lemref{exptails} and a local central limit theorem (\cite{Sto65}
as applied in \cite[(4.36)]{GRZ17}) and the third is by \lemref{endpointsok}.
\end{proof}
We also need a slightly different version of the bound in \propref{hittingprob}:
\begin{prop}
\label{prop:hittingprob-late}There is a constant $C$ so that
\[
\widetilde{\mathbb{P}}_{W,\widetilde{W}}^{x,\tilde{x}}\left[\inf_{\substack{r,\tilde{r}>s\\
|r-\tilde{r}|\le1
}
}|W_{r}-\widetilde{W}_{\tilde{r}}|\le1\right]\le Cs^{-d/2+1}.
\]
\end{prop}

\begin{proof}
Recall the definition \eqref{kappadef} of $\kappa_{2}$ and put $n_{0}=\frac{s}{2\kappa_{2}}.$
Again we abbreviate $\mathbb{P}=\widetilde{\mathbb{P}}_{W,\widetilde{W}}^{x,\tilde{x}}$
and let constants change from line to line. We can estimate
\[
\mathbb{P}\left[\inf_{\substack{r,\tilde{r}>s\\
|r-\tilde{r}|\le1
}
}|W_{r}-\widetilde{W}_{\tilde{r}}|\le1\right]\le\mathbb{P}\left[\inf_{\substack{r,\tilde{r}>\sigma_{n_{0}}\\
|r-\tilde{r}|\le1
}
}|W_{r}-\widetilde{W}_{\tilde{r}}|\le1\right]+\mathbb{P}(\sigma_{n_{0}}\ge s).
\]
A simple large-deviations estimate for geometric random variables
yields
\[
\mathbb{P}(\sigma_{n_{0}}\ge s)\le C\e^{-cn_{0}}\le Cs^{1-d/2},
\]
so it suffices to show that
\[
\mathbb{P}\left[\inf_{\substack{r,\tilde{r}>\sigma_{n_{0}}\\
|r-\tilde{r}|\le1
}
}|W_{r}-\widetilde{W}_{\tilde{r}}|\le1\right]\le Cn_{0}^{1-d/2}.
\]
Define
\[
B_{k}=\max_{r,\tilde{r}\in[\sigma_{k},\sigma_{k+2}]}\left(|W_{r}-W_{\sigma_{k}}|+|\widetilde{W}_{\tilde{r}}-\widetilde{W}_{\sigma_{k}}|\right),
\]
so we have
\begin{multline*}
\mathbb{P}\left[\inf_{\substack{r,\tilde{r}>\sigma_{n_{0}}\\
|r-\tilde{r}|\le1
}
}|W_{r}-\widetilde{W}_{\tilde{r}}|\le1\right]\le\sum_{M=0}^{\infty}\sum_{k=n_{0}}^{\infty}\mathbb{P}\left[|W_{\sigma_{k}}-\widetilde{W}_{\sigma_{k}}|\le2^{M}\right]\mathbb{P}[B_{k}\ge2^{M-1}-1]\\
\le C\sum_{M=0}^{\infty}\e^{-c(2^{M-1}-1)}\sum_{k=n_{0}}^{\infty}\frac{2^{Md}}{k^{d/2}}=Cn_{0}^{1-d/2}\sum_{M=0}^{\infty}\e^{-c(2^{M-1}-1)+CMd}\le Cn_{0}^{1-d/2},
\end{multline*}
where the second inequality again uses the local limit theorem of \cite{Sto65}.
\end{proof}

\section{The stationary solution \label{sec:stationarity}}

The strategy of the proof of \thmref{stationarity} is typical for
the construction of a stationary solution to a PDE: we consider the
Cauchy problem with initial data given at time $s=-S$, and pass to
the limit $S\to+\infty$. This lets us obtain a global-in-time solution
to the problem that satisfies appropriate uniform bounds, provided
that the Lyapunov exponent $\lambda=\lambda(\beta)$ is chosen appropriately.
Let $\Psi(s,y;S)$ be the solution to
\begin{equation}
\begin{aligned}\partial_{s}\Psi(s,y;S) & =\frac{1}{2}\Delta\Psi(s,y;S)+(\beta V(s,y)-\lambda)\Psi(s,y;S),\qquad s>-S;\\
\Psi(-S,y;S) & =1.
\end{aligned}
\label{eq:PsiSproblem}
\end{equation}
The heart of the proof of \thmref{stationarity} is the following
proposition:
\begin{prop}
\label{prop:PsiS1PsiS2}If $\beta$ is sufficiently small, then there
exists $\lambda=\lambda(\beta)$ and a constant $C<\infty$ so that,
with this choice of $\lambda$ in \eqref{PsiSproblem}, for any $0\le S_{1}\le S_{2}$,
we have
\[
\mathbf{E}(\Psi(0,y;S_{2})-\Psi(0,y;S_{1}))^{2}\le CS_{1}^{-d/2+1}.
\]
\end{prop}

Before we prove \propref{PsiS1PsiS2}, we show how it implies \thmref{stationarity}.
\begin{proof}[Proof of \thmref{stationarity}.]
For a positive weight $w\in L^{1}(\mathbb{R}^{d})$, consider the
weighted space $L_{w}^{2}(\mathbb{R}^{d})$, with the inner product
\[
\langle f,g\rangle_{L_{w}^{2}(\mathbb{R}^{d})}=\int f(y)\overline{g(y)}w(y)\,\dif y.
\]
By \propref{PsiS1PsiS2} and the stationarity of $V$ in time, we
have
\begin{align}
\mathbf{E}\|\Psi(s,\cdot;S_{1})-\Psi(s,\cdot;S_{2})\|_{L_{w}^{2}(\mathbb{R}^{d})}^{2} & =\int\mathbf{E}|\Psi(0,y;s+S_{1})-\Psi(0,y;s+S_{2})|^{2}w(y)\,\dif y\nonumber \\
 & \le C(s+S_{1})^{-d/2+1}\|w\|_{L^{1}(\mathbb{R}^{d})},\label{eq:L2wCauchy}
\end{align}
and the right-hand side converges to $0$ as $S_{1},S_{2}\to\infty$,
locally uniformly in $s$. Hence, the family $\Psi(s,y;S)$ converges
in $L^{2}(\Omega;L_{w}^{2}(\mathbb{R}^{d}))$, locally uniformly in
$s$, to a limit $\widetilde{\Psi}$. (Here $\Omega$ denotes the
probability space on which $V$ is defined.) The stationarity of $\widetilde{\Psi}$
is standard. The convergence of $\Psi$ to $\widetilde{\Psi}$ locally
in $L^{2}(\Omega;L_{w}^{2}(\mathbb{R}^{d}))$ implies that $\widetilde{\Psi}$
satisfies \eqref{PsitildePDE} in a weak sense almost surely, hence
in a strong sense almost surely by standard parabolic regularity.

To prove the convergence claimed in \eqref{Psiconvergence}, we use
an argument similar to the above. In particular, we note that the
solution $\Psi(s,y)$ to \eqref{Psiproblem} is stationary in $y$,
as is $\widetilde{\Psi}(s,y)$, so for any fixed $y\in\mathbb{R}^{d}$
we have
\begin{align}
\mathbf{E}|\Psi(s,y)-\widetilde{\Psi}(s,y)|^{2}\int w(y')\,\dif y' & =\int\mathbf{E}|\Psi(s,y')-\widetilde{\Psi}(s,y')|^{2}w(y')\,\dif y'\nonumber \\
 & =\int\mathbf{E}|\Psi(0,y';s)-\widetilde{\Psi}(0,y')|^{2}w(y')\,\dif y,\label{eq:convergetostationary}
\end{align}
and the right-hand side is bounded by a constant times $s^{-d/2+1}$
as $s\to\infty$ by the definition of $\widetilde{\Psi}$. (In the
second equality of \eqref{convergetostationary}, we used the time-stationarity
of $(V,\widetilde{\Psi})$.)
\end{proof}
We also record the covariance kernel of the stationary solution.
\begin{cor}
\label{cor:Psicov}We have
\[
\mathbf{E}[\widetilde{\Psi}(s,y)\widetilde{\Psi}(s,\tilde{y})]=\e^{2\alpha_{\infty}}\widetilde{\mathbb{E}}_{W,\widetilde{W}}^{y,\tilde{y}}\exp\left\{ \beta^{2}\mathscr{R}_{\infty}[W,\widetilde{W}]\right\} .
\]
\end{cor}

In the remainder of this section, we set about proving \propref{PsiS1PsiS2}.
The proof will rely on the Feynman--Kac formula. We recall the Feynman--Kac
formula for $\Psi(s,y;S)$, which comes from \eqref{PsiFKpreliminary}
by a simple time-change:

\begin{equation}
\Psi(s,y;S)=\mathbb{E}_{B}^{y}\exp\left\{ \beta\mathscr{V}_{s;s+S}[B]-\lambda(s+S)\right\} .\label{eq:PsiSdef}
\end{equation}
We first note that spatial stationarity allow us to take $y=0$, and
then the same computation that leads to \eqref{Psicov} gives
\begin{align}
\mathbf{E} & (\Psi(0,0;S_{2})-\Psi(0,0;S_{1}))^{2}\nonumber \\
 & =\e^{2\alpha_{S_{2}}}\widehat{\mathbb{E}}_{B,\widetilde{B};S_{2}}^{0,0}\exp\left\{ \beta^{2}\mathscr{R}_{S_{2}}[B,\widetilde{B}]\right\} -2\e^{\alpha_{S_{2}}+\alpha_{S_{1}}}\widehat{\mathbb{E}}_{B,\widetilde{B};S_{2},S_{1}}^{0,0}\exp\left\{ \beta^{2}\mathscr{R}_{S_{2},S_{1}}[B,\widetilde{B}]\right\} \nonumber \\
 & \qquad+\e^{2\alpha_{S_{1}}}\widehat{\mathbb{E}}_{B,\widetilde{B};S_{2},S_{1}}^{0,0}\exp\left\{ \beta^{2}\mathscr{R}_{S_{1}}[B,\widetilde{B}]\right\} .\label{eq:Psidiffvariance}
\end{align}
Let us now explain intuitively why the right-hand side of this expression
should be small. First, we recall that $\alpha_{s}$ has a limit as
$s\to\infty$ by \eqref{alphaconverges}. Second, as we have observed
in \subsecref{intersections}, in dimension $d\ge3$ two Brownian
motions will almost surely spend at most a finite amount of time within
distance $1$ of each other. In fact, the amount of time they spend
within distance $1$ of each other has (some but not all) exponential
moments. Only such times contribute to $\mathscr{R}_{\bullet}[B,\widetilde{B}]$.
The thrust of \secref{GRZreview} above was that the tilted Brownian
motion, on large scales, again looks like a Brownian motion. This
makes it plausible that, under the tilted measure, the exponential
moments of $\mathscr{R}_{S_{1}}[B,\widetilde{B}]$, $\mathscr{R}_{S_{2},S_{1}}[B,\widetilde{B}]$,
and $\mathscr{R}_{S_{2}}[B,\widetilde{B}]$ are all close to each
other, making the right-hand side of \eqref{Psidiffvariance} small
as $S_{1},S_{2}\to\infty$.

In the rest of this section, we make this reasoning precise. We emphasize
that the computation that we will do still has content in the case
when $V$ is white in time; in this case the tilting has no effect
and $B$ and $\widetilde{B}$ are simply Brownian motions. In that
case, the approximations from \secref{GRZreview} are unnecessary
and the previous paragraph is essentially a proof. Nonetheless, the
reader may find it helpful on first reading to pretend that $B$
and $\widetilde{B}$ are Brownian motions. (In this case the computation
is very similar to that of \cite{MSZ16}.)

Our first lemma is the workhorse of the argument. It makes the above
intuition, which is standard for the Brownian motion, precise for
the case of the Markov chain.
\begin{lem}
\label{lem:Rsstilde}There exists a constant $C<\infty$ so that for
all $\beta$ sufficiently small, the following holds. If $1\le s\le s'\le\tilde{s}\le\tilde{s}'$,
then
\[
\widetilde{\mathbb{E}}_{W,\widetilde{W}}^{y,\tilde{y}}\left|\exp\left\{ \beta^{2}\mathscr{R}_{\tilde{s},\tilde{s}'}[W,\widetilde{W}]\right\} -\exp\left\{ \beta^{2}\mathscr{R}_{s,s'}[W,\widetilde{W}]\right\} \right|\le C(s-1)^{1-d/2}.
\]
\end{lem}

\begin{proof}
We have
\begin{align*}
\widetilde{\mathbb{E}}_{W,\widetilde{W}}^{y,\tilde{y}} & \left|\exp\left\{ \beta^{2}\mathscr{R}_{\tilde{s},\tilde{s}'}[W,\widetilde{W}\right\} -\exp\left\{ \beta^{2}\mathscr{R}_{s,s'}[W,\widetilde{W}]\right\} \right|\\
 & \le\widetilde{\mathbb{E}}_{W,\widetilde{W}}^{y,\tilde{y}}\left|\exp\left\{ \beta^{2}\mathscr{R}_{\infty}[W,\widetilde{W}]\right\} -\exp\left\{ \beta^{2}\mathscr{R}_{s}[W,\widetilde{W}]\right\} \right|\\
 & \le\widetilde{\mathbb{E}}_{W,\widetilde{W}}^{y,\tilde{y}}\exp\left\{ \beta^{2}\mathscr{R}_{\infty}[W,\widetilde{W}]\right\} \mathbf{1}\{\mathscr{R}_{\infty}[W,\widetilde{W}]\ne\mathscr{R}_{S}[W,\widetilde{W}]\}\\
 & \le\widetilde{\mathbb{E}}_{W,\widetilde{W}}^{y,\tilde{y}}\exp\left\{ \beta^{2}\mathscr{R}_{\infty}[W,\widetilde{W}]\right\} \mathbf{1}\{(\exists r,\tilde{r}\ge s-1)\;|r-\tilde{r}|\le2\text{ and }|W_{r}-\tilde{W}_{\tilde{r}}|\le1\}.
\end{align*}
On the event that $\{\mathscr{R}_{\infty}[W,\widetilde{W}]\ne\mathscr{R}_{s}[W,\widetilde{W}]\}$,
let $\tau<\tilde{\tau}$ be the first pair of times after $s-1$ such
that $|\tau-\tilde{\tau}|\le2$ and $|W_{\tau}-\widetilde{W}_{\tilde{\tau}}|\le1$.
Then we have
\begin{align*}
\widetilde{\mathbb{E}}_{W,\widetilde{W}}^{y,\tilde{y}} & \left|\exp\left\{ \beta^{2}\mathscr{R}_{\tilde{s},\tilde{s}'}[W,\widetilde{W}\right\} -\exp\left\{ \beta^{2}\mathscr{R}_{s,s'}[W,\widetilde{W}]\right\} \right|\\
 & \le\int_{s-1}^{\infty}\int_{r}^{r+2}\widetilde{\mathbb{E}}_{W,\widetilde{W}}^{y,\tilde{y}}\left[\exp\left\{ \beta^{2}\mathscr{R}_{\infty}[W,\widetilde{W}]\right\} \;\middle|\;\tau=r,\tilde{\tau}=\tilde{r}\right]\,\dif\widetilde{\mathbb{P}}_{W,\widetilde{W}}^{y,\tilde{y}}(\tau=r,\tilde{\tau}=\tilde{r})\\
 & \le C\widetilde{\mathbb{P}}_{W,\widetilde{W}}^{y,\tilde{y}}\left((\exists r,\tilde{r}\ge s-1)\;|r-\tilde{r}|\le2\text{ and }|W_{r}-\tilde{W}_{\tilde{r}}|\le1\right)\le C(s-1)^{1-d/2},
\end{align*}
where the second inequality is by \propref{expexpbdd-bridge} and
the last is by \propref{hittingprob-late}.
\end{proof}
Now we can prove \propref{PsiS1PsiS2}. The proof combines \lemref{Rsstilde}
with various error bounds from \secref{GRZreview}.
\begin{proof}[Proof of \propref{PsiS1PsiS2}.]
We first re-write \eqref{Psidiffvariance} in terms of the Markov
chain using \eqref{Markovchainexpectation}:
\begin{align}
\mathbf{E} & (\Psi(0,0;S_{2})-\Psi(0,0;S_{1}))^{2}=\widetilde{\mathbb{E}}_{W,\widetilde{W}}\e^{2\alpha_{S_{2}}}\exp\left\{ \beta^{2}\mathscr{R}_{S_{2}}[W,\widetilde{W}]\right\} \mathscr{G}[w_{\lfloor S_{2}\rfloor-1}]\mathscr{G}[\tilde{w}_{\lfloor S_{2}\rfloor-1}]\nonumber \\
 & -2\e^{\alpha_{S_{2}}+\alpha_{S_{1}}}\widetilde{\mathbb{E}}_{W,\widetilde{W}}\exp\left\{ \beta^{2}\mathscr{R}_{S_{2},S_{1}}[W,\widetilde{W}]\right\} \mathscr{G}[w_{\lfloor S_{2}\rfloor-1}]\mathscr{G}[\tilde{w}_{\lfloor S_{1}\rfloor-1}]\nonumber \\
 & +\e^{2\alpha_{S_{1}}+\alpha_{S_{1}}}\widetilde{\mathbb{E}}_{W,\widetilde{W}}\exp\left\{ \beta^{2}\mathscr{R}_{S_{1}}[W,\widetilde{W}]\right\} \mathscr{G}[w_{\lfloor S_{1}\rfloor-1}]\mathscr{G}[\tilde{w}_{\lfloor S_{1}\rfloor-1}].\label{eq:vardiffMC}
\end{align}
For any $S_{1}\le S_{2}$ we can decompose
\begin{align}
 & \widetilde{\mathbb{E}}_{W,\widetilde{W}}\e^{\alpha_{S_{2}}+\alpha_{S_{1}}}\exp\left\{ \beta^{2}\mathscr{R}_{S_{2},S_{1}}[W,\widetilde{W}]\right\} \mathscr{G}[w_{\lfloor S_{2}\rfloor-1}]\mathscr{G}[\tilde{w}_{\lfloor S_{1}\rfloor-1}]\nonumber \\
 & \quad=\e^{2\alpha_{\infty}}\widetilde{\mathbb{E}}_{W,\widetilde{W}}\exp\left\{ \beta^{2}\mathscr{R}_{\frac{9}{10}S_{2},\frac{9}{10}S_{1}}[W,\widetilde{W}]\right\} \nonumber \\
 & \qquad+\e^{2\alpha_{\infty}}\widetilde{\mathbb{E}}_{W,\widetilde{W}}\exp\left\{ \beta^{2}\mathscr{R}_{\frac{9}{10}S_{2},\frac{9}{10}S_{1}}[W,\widetilde{W}]\right\} \left(\mathscr{G}[w_{\lfloor S_{2}\rfloor-1}]\mathscr{G}[\tilde{w}_{\lfloor S_{1}\rfloor-1}]-1\right)\nonumber \\
 & \qquad+\e^{2\alpha_{\infty}}\widetilde{\mathbb{E}}_{W,\widetilde{W}}\left(\exp\left\{ \beta^{2}\mathscr{R}_{S_{2},S_{1}}[W,\widetilde{W}]\right\} -\exp\left\{ \beta^{2}\mathscr{R}_{\frac{9}{10}S_{2},\frac{9}{10}S_{1}}[W,\widetilde{W}]\right\} \right)\mathscr{G}[w_{\lfloor S_{2}\rfloor-1}]\mathscr{G}[\tilde{w}_{\lfloor S_{1}\rfloor-1}]\nonumber \\
 & \qquad+\left(\e^{\alpha_{S_{2}}+\alpha_{S_{1}}}-\e^{2\alpha_{\infty}}\right)\widetilde{\mathbb{E}}_{W,\widetilde{W}}\exp\left\{ \beta^{2}\mathscr{R}_{S_{2},S_{1}}[W,\widetilde{W}]\right\} \mathscr{G}[w_{\lfloor S_{2}\rfloor-1}]\mathscr{G}[\tilde{w}_{\lfloor S_{1}\rfloor-1}].\label{eq:S1S2decomp}
\end{align}
Now \eqref{alphaconverges}, \lemref{Rsstilde}, and \propref{expexpbdd-bridge}
allow us to control the last term of \eqref{S1S2decomp}:
\[
\lim_{S_{1},S_{2}\to\infty}S_{2}^{d/2-1}\left(\e^{\alpha_{S_{2}}+\alpha_{S_{1}}}-\e^{2\alpha_{\infty}}\right)\widetilde{\mathbb{E}}_{W,\widetilde{W}}\exp\left\{ \beta^{2}\mathscr{R}_{S_{2},S_{1}}[W,\widetilde{W}]\right\} \mathscr{G}[w_{\lfloor S_{2}\rfloor-1}]\mathscr{G}[\tilde{w}_{\lfloor S_{1}\rfloor-1}]=0.
\]
\lemref{Rsstilde} also allows us to bound the third term of \eqref{S1S2decomp}:
\begin{align*}
 & S_{1}^{\frac{d}{2}-1}\left|\widetilde{\mathbb{E}}_{W,\widetilde{W}}\left(\exp\left\{ \beta^{2}\mathscr{R}_{S_{2},S_{1}}[W,\widetilde{W}]\right\} -\exp\left\{ \beta^{2}\mathscr{R}_{\frac{9}{10}S_{2},\frac{9}{10}S_{1}}[W,\widetilde{W}]\right\} \right)\mathscr{G}[w_{\lfloor S_{2}\rfloor-1}]\mathscr{G}[\tilde{w}_{\lfloor S_{1}\rfloor-1}]\right|\\
 & \qquad\le\|\mathscr{G}\|_{\infty}^{2}(S_{1}\wedge S_{2})^{d/2-1}\widetilde{\mathbb{E}}_{W,\widetilde{W}}\left(\exp\left\{ \beta^{2}\mathscr{R}_{S_{2},S_{1}}[W,\widetilde{W}]\right\} -\exp\left\{ \beta^{2}\mathscr{R}_{\frac{9}{10}S_{2},\frac{9}{10}S_{1}}[W,\widetilde{W}]\right\} \right)\le C,
\end{align*}
for a constant $C$ independent of $S_{1}$ and $S_{2}$. For the
second term of \eqref{S1S2decomp}, we can use \lemref{cutofftail}
to get
\[
\limsup_{S_{1},S_{2}\to\infty}S_{1}^{d/2-1}\widetilde{\mathbb{E}}_{W,\widetilde{W}}\exp\left\{ \beta^{2}\mathscr{R}_{\frac{9}{10}S_{2},\frac{9}{10}S_{1}}[W,\widetilde{W}]\right\} \left(\mathscr{G}[w_{\lfloor S_{2}\rfloor-1}]\mathscr{G}[\tilde{w}_{\lfloor S_{1}\rfloor-1}]-1\right)=0.
\]
Finally, we have that
\[
S_{1}^{d/2-1}\widetilde{\mathbb{E}}_{W,\widetilde{W}}\left[\exp\left\{ \beta^{2}\mathscr{R}_{\frac{9}{10}S_{2}}[W,\widetilde{W}]\right\} -2\exp\left\{ \beta^{2}\mathscr{R}_{\frac{9}{10}S_{2},\frac{9}{10}S_{1}}[W,\widetilde{W}]\right\} +\exp\left\{ \beta^{2}\mathscr{R}_{\frac{9}{10}S_{1}}[W,\widetilde{W}]\right\} \right]
\]
is bounded above independently of $S_{1}$ and $S_{2}$, also by \lemref{Rsstilde}.
Substituting \eqref{S1S2decomp} into \eqref{vardiffMC}, and then
applying the last four bounds, we see that
\[
\mathbf{E}(\Psi(0,y;S_{2})-\Psi(0,y;S_{1}))^{2}\le CS_{1}^{-d/2+1},
\]
as claimed.
\end{proof}

\section{The effective noise strength\label{sec:noisestrength}}

In this section, we explain how the effective noise strength parameter
$\nu$ in \eqref{EWPDE} arises from the stationary solution $\widetilde{\Psi}$
and prove \thmref{noiseexpr}.
\begin{lem}
\label{lem:varuniformconv}If $\beta$ is sufficiently small and $g\in\mathcal{C}_{\mathrm{c}}^{\infty}(\mathbb{R}^{d})$,
then we have
\[
\lim_{t\to\infty}\left|\Var\left(\eps^{-d/2+1}\int g(x)\Psi(\eps^{-2}t,\eps^{-1}x)\,\dif x\right)-\Var\left(\eps^{-d/2+1}\int g(x)\widetilde{\Psi}(0,\eps^{-1}x)\,\dif x\right)\right|=0,
\]
uniformly in $\eps>0$.
\end{lem}

\begin{proof}
We have
\begin{align*}
 & \left|\Var\left(\eps^{-d/2+1}\int g(x)\Psi(\eps^{-2}t,\eps^{-1}x)\,\dif x\right)-\Var\left(\eps^{-d/2+1}\int g(x)\widetilde{\Psi}(0,\eps^{-1}x)\,\dif x\right)\right|\\
 & \qquad\le\eps^{-d+2}\|g\|_{L^{1}(\mathbb{R}^{d})}\int|g(x)|\mathbf{E}\left|\Psi(\eps^{-2}t,\eps^{-1}x)-\widetilde{\Psi}(\eps^{-2}t,\eps^{-1}x)\right|^{2}\,\dif x\\
 & \qquad\le C\eps^{-d+2}\|g\|_{L^{1}(\mathbb{R}^{d})}^{2}(\eps^{-2}t)^{-d/2+1}\le C\|g\|_{L^{1}(\mathbb{R}^{d})}^{2}t^{-d/2+1},
\end{align*}
where the first inequality is by the time-stationarity of $\widetilde{\Psi}$
and Jensen's inequality and the second is by \eqref{Psiconvergence}.
\end{proof}
We recall from \cite[Lemmas 3.1, 3.2 and 3.3]{GRZ17} that
\begin{equation}
\lim_{\eps\to0}\Var\left(\frac{\e^{-\alpha_{t/\eps^{2}}}}{\eps^{d/2-1}}\int g(x)\Psi(\eps^{-2}t,\eps^{-1}x)\,\dif x\right)=\Var\left(\int g(x)\psi(t,x)\,\dif x\right),\label{eq:distconv}
\end{equation}
where $\psi$ is the solution to the Edwards-Wilkinson stochastic partial differential equation
\begin{equation}
\begin{aligned}\partial_{t}\psi & =\frac{1}{2}a\Delta\psi+\beta\nu\dot{W},\qquad t>0,x\in\mathbb{R}^{d};\\
\psi(0,x) & =0,
\end{aligned}
\label{eq:EWPDE-1}
\end{equation}
which is simply \eqref{EWPDE} with $\overline{u}\equiv1$.
\begin{lem}
\label{lem:spdestat}We have
\begin{equation}
\lim_{t\to\infty}\Var\left(\int g(x)\psi(t,x)\,\dif x\right)=\beta^{2}\nu^{2}\int_{0}^{\infty}\int|\overline{g}(r,x)|^{2}\,\dif x\,\dif r,\label{eq:SPDEvar1}
\end{equation}
where $\overline{g}$ is the solution of
\begin{align*}
\partial_{t}\overline{g}(t,x) & =\frac{1}{2}a\Delta\overline{g}(t,x),\qquad t>0,x\in\mathbb{R}^{d};\\
\overline{g}(0,x) & =g(x).
\end{align*}
\end{lem}

\begin{proof}
As in \cite[(3.16)]{GRZ17}, we have
\begin{equation*}
\Var\left(\int g(x)\psi(t,x)\,\dif x\right)=\beta^{2}\nu^{2}\int_{0}^{t}\int|\overline{g}(t-r,x)|^{2}\,\dif x\,\dif r =\beta^{2}\nu^{2}\int_{0}^{t}\int|\overline{g}(r,x)|^{2}\,\dif x\,\dif r,%
\end{equation*}
and then the result follows by taking $t\to\infty$. %
\end{proof}
Now we are ready to prove \thmref{noiseexpr}.
\begin{proof}[Proof of \thmref{noiseexpr}.]
Fix $\delta>0$. By \lemref[s]{varuniformconv}~and~\ref{lem:spdestat},
we can choose $t$ large enough, independently of $\eps$, so that
\[
\left|\Var\left(\int g(x)\psi(t,x)\,\dif x\right)-\beta^{2}\nu^{2}\int_{0}^{\infty}\int|\overline{g}(r,x)|^{2}\,\dif x\,\dif r\right|<\delta/3
\]
and
\[
\left|\Var\left(\eps^{-d/2+1}\int g(x)\Psi(\eps^{-2}t,\eps^{-1}x)\,\dif x\right)-\Var\left(\eps^{-d/2+1}\int g(x)\widetilde{\Psi}(\eps^{-2}t,\eps^{-1}x)\,\dif x\right)\right|<\delta/3.
\]
Then by \eqref{distconv} we can choose $\eps$ so small that
\[
\left|\Var\left(\eps^{-\alpha_{t/\eps^{2}}}\eps^{-d/2+1}\int g(x)\Psi(\eps^{-2}t,\eps^{-1}x)\,\dif x\right)-\Var\left(\int g(x)\psi(t,x)\,\dif x\right)\right|<\delta/3.
\]
Using the triangle inequality on the last three expressions, and recalling
\eqref{alphaconverges}, we obtain
\begin{equation}
\lim_{\eps\to0}\Var\left(\eps^{-d/2+1}\int g(x)\widetilde{\Psi}(0,\eps^{-1}x)\,\dif x\right)=\e^{2\alpha_{\infty}}\beta^{2}\nu^{2}\int_{0}^{\infty}\int|\overline{g}(r,x)|^{2}\,\dif x\,\dif r.\label{eq:usetrieq}
\end{equation}
The left-hand side of \eqref{usetrieq} is equal to 
\[
\lim_{\eps\to0}\int\int g(x)g(\tilde{x})\eps^{-d+2}\Cov\left(\widetilde{\Psi}(0,\eps^{-1}x),\widetilde{\Psi}(0,\eps^{-1}\tilde{x})\right)\,\dif x\,\dif\tilde{x},
\]
while the right-hand side of \eqref{usetrieq} is equal to
\begin{align*}
\e^{2\alpha_{\infty}}\beta^{2}\nu^{2} & \int\int\left(\int_{0}^{\infty}\int G_{a}(r,z-x)G_{a}(r,z-\tilde{x})\,\dif z\,\dif r\right)g(x)g(\tilde{x})\,\dif x\,\dif\tilde{x}\\
 & =\e^{2\alpha_{\infty}}\beta^{2}\nu^{2}ca^{-1}\int\int|x-\tilde{x}|^{-d+2}g(x)g(\tilde{x})\,\dif x\,\dif\tilde{x},
\end{align*}
where $G_{a}$ and $c$ are defined as in \eqref{Gadef}--\eqref{cbardef}.
Therefore, we have
\[
\nu^{2}=\frac{a\lim\limits _{\eps\to0}\int\int g(x)g(\tilde{x})\eps^{-d+2}\Cov\left(\widetilde{\Psi}(0,\eps^{-1}x),\widetilde{\Psi}(0,\eps^{-1}\tilde{x})\right)\,\dif x\,\dif\tilde{x}}{c\e^{2\alpha_{\infty}}\beta^{2}\int\int g(x)g(\tilde{x})|x-\tilde{x}|^{-d+2}\,\dif x\,\dif\tilde{x}},
\]
which is \eqref{nu2expr}.
\end{proof}

\section{The effective diffusivity\label{sec:effective-diffusivity}}

In this section we explain how to relate the effective diffusivity
$a$ to the asymptotic expansion \eqref{expansion}.

\subsection{The solvability condition}

We first explain how the effective diffusivity $a$ can be recovered
formally from the homogenization correctors for \eqref{expansion}. We define these correctors now, and for the moment we disregard the question of their existence. %
We
start with the equations \eqref{u1problem}--\eqref{u2problem} for
the terms $u_{1}$ and $u_{2}$ in the formal asymptotic expansion
\eqref{expansion} for $u^{\eps}$. We will replace $\Psi$ on the
right-hand side of these equations by the stationary solution $\widetilde{\Psi}$,
so our formal starting point is
\begin{equation}
\partial_{s}u_{1}(t,x,s,y)=\frac{1}{2}\Delta_{y}u_{1}(t,x,s,y)+(\beta V(s,y)-\lambda)u_{1}(t,x,s,y)+\nabla_{y}\widetilde{\Psi}(s,y)\cdot\nabla_{x}\overline{u}(t,x)\label{eq:u1problem-1-1}
\end{equation}
and
\begin{equation}
\begin{aligned}\partial_{s}u_{2}(t,x,s,y) & =\frac{1}{2}\Delta_{y}u_{2}(t,x,s,y)+(\beta V(s,y)-\lambda)u_{2}(t,x,s,y)+\nabla_{y}\cdot\nabla_{x}u_{1}(t,x,s,y)\\
 & \qquad+\frac{1}{2}(1-a)\widetilde{\Psi}(s,y)\Delta_{x}\overline{u}(t,x).
\end{aligned}
\label{eq:u2problem-1-1}
\end{equation}
We can now formally decompose the solution to \eqref{u1problem-1-1}
as
\begin{equation}
u_{1}(t,x,s,y)=\tilde{\boldsymbol{\omega}}(s,y)\cdot\nabla_{x}\overline{u}(t,x),\label{eq:u1decomp}
\end{equation}
where $\tilde{\boldsymbol{\omega}}(s,y)=(\tilde{\omega}^{(1)}(s,y),\ldots,\tilde{\omega}^{(d)}(s,y))$
is a space-time-stationary solution to
\begin{equation}
\partial_{s}\tilde{\omega}^{(k)}=\frac{1}{2}\Delta_{y}\tilde{\omega}^{(k)}+(\beta V-\lambda)\tilde{\omega}^{(k)}+\frac{\partial\widetilde{\Psi}}{\partial y_{k}}.\label{eq:omegaproblem}
\end{equation}
We note that, unlike the random heat equation \eqref{PsitildePDE},
the forced equation \eqref{omegaproblem} may not have stationary
solutions in all $d\ge3$. Nevertheless, the formal computation will
give us an idea of how the effective diffusivity can be approximated.
By \thmref{stationarity}, applied with time reversed (or equivalently
to the random heat equation with potential $V(-s,y)$), we also have
a stationary solution $\widetilde{\Phi}$ to the equation
\begin{equation}
-\partial_{s}\widetilde{\Phi}=\frac{1}{2}\Delta\widetilde{\Phi}+(\beta V-\lambda)\widetilde{\Phi}.\label{eq:Phitildepde}
\end{equation}
Multiplying \eqref{u2problem-1-1} by $\Phi$ and using \eqref{u1decomp}
and \eqref{Phitildepde} gives
\begin{equation}
\begin{aligned}\partial_{s}(\widetilde{\Phi}(s,y)u_{2}(t,x,s,y)) & =\frac{1}{2}\widetilde{\Phi}(s,y)\Delta_{y}u_{2}(t,x,s,y)-\frac{1}{2}u_{2}(t,x,s,y)\Delta\widetilde{\Phi}(s,y)\\
 & \qquad+\widetilde{\Phi}(s,y)\tr(\nabla_{y}\tilde{\boldsymbol{\omega}}(s,y)\cdot\Hess\overline{u}(t,x))+\frac{1}{2}(1-a)\widetilde{\Phi}(s,y)\widetilde{\Psi}(s,y)\Delta_{x}\overline{u}(t,x).
\end{aligned}
\label{eq:correctoreqn}
\end{equation}
The assumed stationarity of $u_{2}$ in $s$ and the stationarity
of $\widetilde{\Phi}$ in $s$ imply that the expectation of the left-hand
side is $0$. Stationarity of $u_{2}$ in $y$, on the other hand,
implies that
\[
\mathbf{E}\left[\widetilde{\Phi}(s,y)\Delta_{y}u_{2}(t,x,s,y)-u_{2}(t,x,s,y)\Delta\widetilde{\Phi}(s,y)\right]=0.
\]
Therefore, taking the expectation of \eqref{correctoreqn} yields
\[
\mathbf{E}\widetilde{\Phi}(s,y)\left[\tr(\nabla_{y}\tilde{\boldsymbol{\omega}}(s,y)\cdot\Hess\overline{u}(t,x))+\frac{1}{2}(1-a)\widetilde{\Psi}(s,y)\Delta_{x}\overline{u}(t,x)\right]=0.
\]
Due to the assumption of isotropy, we have
\[
\mathbf{E}\widetilde{\Phi}\nabla_{y}\tilde{\boldsymbol{\omega}}=\frac{1}{d}\tr\left(\mathbf{E}\widetilde{\Phi}\nabla_{y}\tilde{\boldsymbol{\omega}}\right)I_{d\times d}=\frac{1}{d}\mathbf{E}\widetilde{\Phi}(\nabla_{y}\cdot\tilde{\boldsymbol{\omega}})I_{d\times d},
\]
and thus
\begin{align*}
0 & =\mathbf{E}\widetilde{\Phi}(s,y)\left[\tr(\nabla_{y}\tilde{\boldsymbol{\omega}}(s,y)\cdot\Hess\overline{u}(t,x))+\frac{1}{2}(1-a)\widetilde{\Psi}(s,y)\Delta_{x}\overline{u}(t,x)\right]\\
 & =\mathbf{E}\widetilde{\Phi}(s,y)\left[\frac{1}{d}\nabla_{y}\cdot\tilde{\boldsymbol{\omega}}(s,y)+\frac{1}{2}(1-a)\widetilde{\Psi}(s,y)\right]\Delta_{x}\overline{u}(t,x),
\end{align*}
leading to
\begin{equation}
a=1+\frac{2}{d}\frac{\mathbf{E}[\widetilde{\Phi}(s,y)\nabla_{y}\cdot\tilde{\boldsymbol{\omega}}(s,y)]}{\mathbf{E}[\widetilde{\Phi}(s,y)\widetilde{\Psi}(s,y)]}.\label{eq:oura}
\end{equation}
As we have not proved that a stationary corrector $\tilde{\boldsymbol{\omega}}$
actually exists, the expression \eqref{oura} is purely formal. In
the next section, we will explain how we can use an approximate version
of $\tilde{\boldsymbol{\omega}}$ to write a rigorous version of the
computation leading to \eqref{oura}.

\subsection{An approximation of the effective diffusivity}

In this section, we will show how approximate correctors can be used
in the right-hand side of \eqref{oura} to provide a good approximation
of the effective diffusivity. Instead of trying to build a stationary
solution to the corrector equation \eqref{omegaproblem}, we take
$S>0$ and consider the the solution $\boldsymbol{\omega}(s,y;S)$
of the Cauchy problem for \eqref{omegaproblem}, with $\widetilde{\Psi}(s,y)$
replaced by $\Psi(s,y;S)$ (defined in \eqref{PsiSproblem}):
\begin{align*}
\partial_{s}\omega^{(k)}(s,y) & =\frac{1}{2}\Delta_{y}\omega^{(k)}(s,y)+(\beta V(s,y)-\lambda)\omega^{(k)}(s,y)+\frac{\partial\Psi(s,y;S)}{\partial y_{k}},\qquad s>-S,\ k=1,\ldots,d;\\
\boldsymbol{\omega}(-S,\cdot;S) & \equiv0.
\end{align*}
The solution is given by the Feynman--Kac formula
\begin{equation}
\boldsymbol{\omega}(s,y;S)=\mathbb{E}_{B}^{y}\left[\int_{0}^{s+S}\exp\left\{ \beta\mathscr{V}_{s;[0,r]}[B]-\lambda r\right\} \nabla\Psi(s-r,B_{r};S)\,\dif r\right].\label{eq:omegaSFK}
\end{equation}
We also define, similarly to the definition \eqref{PsiSproblem}/\eqref{PsiSdef}
of $\Psi(s,y;S)$, the 
\begin{equation}
\Phi(s,y;T)=\mathbb{E}_{B}^{y}\exp\left\{ \beta\mathscr{V}_{s;[s-T,0]}[B]-\lambda(T-s)\right\} ,\qquad s<T,\label{eq:PhiFK}
\end{equation}
which solves \eqref{Phitildepde} with terminal condition
\[
\Phi(T,y;T)=1.
\]
Recall that $\mathscr{V}$ was defined in \eqref{Vsadef}, so in particular
we have
\begin{align*}
\mathscr{V}_{s;[0,r]}[B] & =\int_{0}^{r}V(s-\tau,B_{\tau})\,\dif\tau;\\
\mathscr{V}_{s;[s-T,0]} & =\int_{s-T}^{0}V(s-\tau,B_{\tau})\,\dif\tau.
\end{align*}
Note that in the second expression we are evaluating $B$ at negative
times, interpreting it as a two-sided Brownian motion. Now we define
an approximate version of \eqref{oura}.
\begin{equation}
a_{S,T}(s,y)=1+\frac{2}{d}\frac{\mathbf{E}[\Phi(s,y;T)\nabla_{y}\cdot\boldsymbol{\omega}(s,y;S)]}{\mathbf{E}[\Phi(s,y;T)\Psi(s,y;S)]}.\label{eq:aSTdef}
\end{equation}
The next theorem, which is the main result of this section, shows
that the ``large $S,T$'' limit of \eqref{aSTdef} agrees with the
effective diffusivity from \eqref{adef} (established in \cite{GRZ17}).
\begin{thm}
\label{thm:aSTthm}Let $a$ be the effective diffusivity defined by
\eqref{adef}. Then we have, for each $s\in\mathbb{R}$ and $y\in\mathbb{R}^{d}$,
\[
\lim_{\substack{S\to\infty\\
T\to\infty
}
}a_{S,T}(s,y)=a.
\]
\end{thm}

We note that if a stationary $\tilde{\boldsymbol{\omega}}$ given
by 
\[
\tilde{\boldsymbol{\omega}}(s,y)=\lim_{S\to\infty}\boldsymbol{\omega}(s,y;S)
\]
exists, then \thmref{aSTthm} verifies the formal expression \eqref{oura}.
Such large-scale approximations of the effective diffusivity have
been used in the different context of elliptic homogenization theory;
see \cite{gloriaotto7}.

Without loss of generality, we will take $s=0$ and $y=0$ in the
proof of \thmref{aSTthm}. In the course of the proof, we will denote
by $H(x)$ the standard Heaviside function $H(x)=\mathbf{1}\{x\ge0\}$
and also use its regularization
\[
H_{\gamma}(x)=\begin{cases}
0 & x\le0;\\
\gamma^{-1}x & 0\le x\le\gamma;\\
1 & x\ge\gamma,
\end{cases}
\]
as well as $J(x)=xH(x)$. While several of the following lemmas are written using this regularization, the statement of \thmref{aSTthm}
does not depend on the regularization. (Ultimately we take $\gamma\to0$.)
We begin with a Feynman--Kac formula for the numerator on the right-hand
side of \thmref{aSTthm}. 

\begin{lem}
\label{lem:writeitasaderiv}We have
\begin{equation}
\begin{aligned}\mathbf{E} & \left[\Phi(0,0;T)(\nabla_{y}\cdot\boldsymbol{\omega})(0,0;S)\right]\\
 & =\nabla_{\eta}|_{\eta=0}\cdot\nabla_{\xi}|_{\xi=0}\mathbf{E}\mathbb{E}_{B}^{0}\exp\left\{ \beta\int_{-T}^{S}V(-\tau,B_{\tau}+H(\tau)\eta+J(\tau)\xi)\,\dif\tau-\lambda(T+S)\right\} \\
 & =\nabla_{\eta}|_{\eta=0}\cdot\nabla_{\xi}|_{\xi=0}\mathbb{E}_{B}^{0}\exp\left\{ \frac{1}{2}\beta^{2}\mathscr{R}_{[-T,S]}[B+H\eta+J\xi]-\lambda(T+S)\right\} .
\end{aligned}
\label{eq:writeitasaderiv}
\end{equation}
\end{lem}

\begin{proof}
From \eqref{omegaSFK} and \eqref{PsiSdef}, we have
\begin{align}
\boldsymbol{\omega}(0,y;S) & =\mathbb{E}_{B}^{y}\int_{0}^{S}\exp\left\{ \beta\int_{0}^{r}V(-\tau,B_{\tau})\,\dif\tau-\lambda r\right\} \nabla\Psi(-r,B_{r};S)\,\dif r\nonumber \\
 & =\nabla_{\xi}|_{\xi=0}\mathbb{E}_{B}^{y}\int_{0}^{S}\exp\left\{ \beta\int_{0}^{S}V(-\tau,B_{\tau}+H(\tau-r)\xi)\,\dif\tau-\lambda S\right\} \,\dif r.\label{eq:omegaFKexpansion}
\end{align}
One can check by explicit differentiation of both expressions that
the right-hand side of \eqref{omegaFKexpansion} can be re-written
as
\begin{equation}
\boldsymbol{\omega}(0,y;S)=\nabla_{\xi}|_{\xi=0}\mathbb{E}_{B}^{y}\exp\left\{ \beta\int_{0}^{S}V(-\tau,B_{\tau}+\tau\xi)\,\dif\tau-\lambda S\right\} .\label{eq:omegaSFKfirstcomp}
\end{equation}
Taking the divergence and setting $y=0$, we can write
\[
(\nabla_{y}\cdot\boldsymbol{\omega})(0,0;S)=\nabla_{\eta}|_{\eta=0}\cdot\nabla_{\xi}|_{\xi=0}\mathbb{E}_{B}^{0}\exp\left\{ \beta\int_{0}^{S}V(-\tau,B_{\tau}+\eta+\tau\xi)\,\dif\tau-\lambda S\right\} .
\]
Multiplying by \eqref{PhiFK} gives
\begin{align*}
\Phi(0,0;T)(\nabla_{y}\cdot\boldsymbol{\omega}) & (0,0;S)\\
 & =\nabla_{\eta}|_{\eta=0}\cdot\nabla_{\xi}|_{\xi=0}\mathbb{E}_{B}^{0}\exp\left\{ \beta\int_{-T}^{S}V(-\tau,B_{\tau}+H(\tau)\eta+J(\tau)\xi)\,\dif\tau-\lambda(T+S)\right\} .
\end{align*}
(Now we are evaluating $B$ at both positive and negative times.)
Taking the expectation 
yields the first equality in \eqref{writeitasaderiv}. The second
inequality then arises from evaluating the expectation.
\end{proof}
It will be useful to write a regularized version of \eqref{writeitasaderiv},
which will later allow us to use the Girsanov formula.
\begin{cor}
We have
\begin{equation}
\begin{aligned}\mathbf{E} & \left[\Phi(0,0;T)(\nabla_{y}\cdot\boldsymbol{\omega})(0,0;S)\right]\\
 & =\lim_{\gamma\downarrow0}\nabla_{\eta}|_{\eta=0}\cdot\nabla_{\xi}|_{\xi=0}\mathbf{E}\mathbb{E}_{B}^{0}\exp\left\{ \int_{-T}^{S}V(-\tau,B_{\tau}+H_{\gamma}(\tau)\eta+J(\tau)\xi)\,\dif\tau-\lambda(T+S)\right\} \\
 & =\nabla_{\eta}|_{\eta=0}\cdot\nabla_{\xi}|_{\xi=0}\mathbb{E}_{B}^{0}\exp\left\{ \frac{1}{2}\beta^{2}\mathscr{R}_{[-T,S]}[B+H_{\gamma}\eta+J\xi]-\lambda(T+S)\right\}.
\end{aligned}
\label{eq:writeitasaderiv-1-2}
\end{equation}
\end{cor}

\begin{proof}
Similarly to \eqref{writeitasaderiv}, the second equality of \eqref{writeitasaderiv-1-2} is a simple
computation, so it suffices to prove that the first expression is
equal to the third. We write out all of the gradients in the third
expression. Define $\delta f(\tau,\tilde{\tau})=f(\tau)-f(\tilde{\tau})$.
For all $\gamma\ge0$ we have
\begin{align}
\nabla_{\eta}|_{\eta=0}\cdot\nabla_{\xi}|_{\xi=0} & \e^{-\lambda(T+S)}\mathbb{E}_{B}^{y}\exp\left\{ \beta^{2}\mathscr{R}_{[-T,S]}[B+H_{\gamma}\eta+J\xi]\right\} \nonumber \\
 & =\beta^{2}\e^{-\lambda(T+S)}\mathbb{E}_{B}^{y}(g_{1;\gamma}[B]+g_{2;\gamma}[B]\cdot g_{3;\gamma}[B])\exp\left\{ \beta^{2}\mathscr{R}_{[-T,S]}[B]\right\} ,\label{eq:lotsofgs}
\end{align}
where we define
\begin{align*}
g_{1;\gamma}[B] & =\iint_{[-2,2]^{2}}\delta H_{\gamma}(\tau,\tilde{\tau})\delta J(\tau,\tilde{\tau})\Delta R(\tau-\tilde{\tau},\delta B(\tau,\tilde{\tau}))\,\dif\tau\,\dif\tilde{\tau},\\
g_{2;\gamma}[B] & =\iint_{[-S,-T]^{2}}\delta J(\tau,\tilde{\tau})\nabla R(\tau-\tilde{\tau},\delta B(\tau,\tilde{\tau}))\,\dif\tau\,\dif\tilde{\tau},\\
g_{3;\gamma}[B] & =\iint_{[-2,2]^{2}}\delta H_{\gamma}(\tau,\tilde{\tau})\nabla R(\tau-\tilde{\tau},\delta B(\tau,\tilde{\tau}))\,\dif\tau\,\dif\tilde{\tau}.
\end{align*}
Here we have used the fact that $R(s,y)=0$ whenever $s\ne[-1,1]$.
Then the bounded convergence theorem implies the right-hand side of
\eqref{lotsofgs} is continuous in $\gamma$, so
\begin{align*}
\nabla_{\eta}|_{\eta=0}\cdot\nabla_{\xi}|_{\xi=0}\lim_{\gamma\downarrow0} & \e^{-\lambda(T+S)}\mathbb{E}_{B}^{y}\exp\left\{ \beta^{2}\mathscr{R}_{[-T,S]}[B+H_{\gamma}\eta+J\xi]\right\} \\
 & =\nabla_{\eta}|_{\eta=0}\cdot\nabla_{\xi}|_{\xi=0}\e^{-\lambda(T+S)}\mathbb{E}_{B}^{y}\exp\left\{ \beta^{2}\mathscr{R}_{[-T,S]}[B+H\eta+J\xi]\right\} ,
\end{align*}
and the result follows from \lemref{writeitasaderiv}.
\end{proof}
\begin{lem}
\label{lem:aSTCM}We have
\begin{equation}
a_{S,T}(0,0)=1+2\lim_{\gamma\downarrow0}\widehat{\mathbb{E}}_{B;[-T,S]}^{0}\left(\frac{1}{\gamma d}B_{S}\cdot B_{\gamma}-1\right).\label{eq:aSTregularized}
\end{equation}
\end{lem}

\begin{proof}
To address the numerator of \eqref{aSTdef}, continue from \eqref{writeitasaderiv-1-2}
and use the Girsanov formula, writing
\begin{align*}
\nabla_{\eta}|_{\eta=0}\cdot & \nabla_{\xi}|_{\xi=0}\mathbb{E}_{B}^{0}\exp\left\{ \beta\int_{-T}^{S}V(-\tau,B_{\tau}+H_{\gamma}(\tau)\eta+J(\tau)\xi)\,\dif\tau-\lambda(T+S)\right\} \\
 & =\nabla_{\eta}|_{\eta=0}\cdot\nabla_{\xi}|_{\xi=0}\mathbb{E}_{B}^{0}\exp\left\{ \beta\int_{-T}^{S}V(-\tau,B_{\tau})\,\dif\tau-\right.\\
 & \qquad\qquad\qquad\qquad\qquad\qquad\qquad\left.-\lambda(T+S)+\frac{1}{\gamma}B_{\gamma}\cdot\eta-\frac{1}{2\gamma}|\eta|^{2}-\xi\cdot\eta+B_{S}\cdot\xi-\frac{1}{2}|\xi|^{2}S\right\} \\
 & =\mathbb{E}_{B}^{0}\left(\frac{1}{\gamma}B_{S}\cdot B_{\gamma}-d\right)\exp\left\{ \beta\int_{-T}^{S}V(-\tau,B_{\tau})\,\dif\tau-\lambda(T+S)\right\} .
\end{align*}
Passing to the limit as $\gamma\downarrow0$ and taking expectations
shows that
\begin{equation}
\mathbf{E}[\Phi(0,0;T)(\nabla\cdot\boldsymbol{\omega})(0,0;S)]=\e^{-\lambda(T+S)}\lim_{\gamma\downarrow0}\mathbb{E}_{B}^{0}\left(\gamma^{-1}B_{S}\cdot B_{\gamma}-d\right)\exp\left\{ \beta^{2}\mathscr{R}_{[-T,S]}[B]\right\} .\label{eq:astnumerator}
\end{equation}

For the denominator of \eqref{aSTdef}, we write
\[
\Phi(0,0;T)\Psi(0,0;S)=\mathbb{E}_{B}^{0}\exp\left\{ \beta\mathscr{V}_{0;[-T,S]}[B]-\lambda(T+S)\right\} 
\]
(where again we use the interpretation of $B$ as a two-sided Brownian
motion), so
\begin{equation}
\mathbf{E}\Phi(0,0;T)\Phi(0,0;S)=\e^{-\lambda(T+S)}\mathbb{E}_{B}^{0}\exp\left\{ \beta^{2}\mathscr{R}_{[-T,S]}[B]\right\} .\label{eq:astdenominator}
\end{equation}
Dividing \eqref{astnumerator} by \eqref{astdenominator} yields \eqref{aSTregularized}.
\end{proof}
\begin{lem}
\label{lem:smallscalesd}We have
\[
\lim_{\gamma\downarrow0}\frac{1}{\gamma d}\widehat{\mathbb{E}}_{B;[-T,S]}^{0}|B_{\gamma}|^{2}=1,
\]
uniformly in $S$ and $T$.
\end{lem}

\begin{proof}
We have
\begin{align*}
\widehat{\mathbb{E}}_{B;[-T,S]}^{0}|B_{\gamma}|^{2}-\mathbb{E}_{B}^{0}|B_{\gamma}|^{2} & =\mathbb{E}_{B}^{0}|B_{\gamma}|^{2}\left(\frac{1}{Z_{[-T,S]}}\exp\left\{ \frac{1}{2}\beta^{2}\mathscr{R}_{[-T,S]}[B]\right\} -1\right)\\
 & =\frac{1}{Z_{[-T,S]}}\mathbb{E}_{B}^{0}|B_{\gamma}|^{2}\left(\exp\left\{ \frac{1}{2}\beta^{2}\mathscr{R}_{[-T,S]}[B]\right\} -\exp\left\{ \frac{1}{2}\beta^{2}\mathscr{R}_{[-T,S]}[\widetilde{B}]\right\} \right),
\end{align*}
where $\widetilde{B}$ is a Brownian motion whose increments on $[-T,0]$
and $[\gamma,S]$ are identical to those of $B$ and whose increments
on $[0,\gamma]$ are independent of those of $B$. (Thus the second
equality is because $\mathscr{R}_{[-T,S]}[\widetilde{B}]$ is independent
of $B_{\gamma}$.) This means that
\begin{align*}
\left|\widehat{\mathbb{E}}_{B;[-T,S]}^{0}|B_{\gamma}|^{2}-\mathbb{E}_{B}^{0}|B_{\gamma}|^{2}\right| & =\frac{1}{Z_{[-T,S]}}\mathbb{E}_{B}^{0}\left(\exp\left\{ \beta^{2}\mathscr{R}_{[-T,0]}[B]\right\} +\exp\left\{ \beta^{2}\mathscr{R}_{[\gamma,S]}[B]\right\} \right)\\
 & \qquad\times\mathbb{E}_{B}^{0}|B_{\gamma}|^{2}\left|\exp\left\{ 2\beta^{2}\int_{-1}^{\gamma}\int_{\tau\vee0}^{1}R(\tau-\tilde{\tau},B_{\tau}-B_{\tilde{\tau}})\,\dif\tilde{\tau}\,\dif\tau\right\} \right.\\
 & \qquad\qquad\qquad\qquad\left.-\exp\left\{ 2\beta^{2}\int_{-1}^{\gamma}\int_{\tau\vee0}^{1}R(\tau-\tilde{\tau},\widetilde{B}_{\tau}-\widetilde{B}_{\tilde{\tau}})\,\dif\tilde{\tau}\,\dif\tau\right\} \right|\\
 & \le C(\mathbb{E}_{B}^{0}|B_{\gamma}|^{4})^{1/2}(\mathbb{E}_{B}^{0}(\exp\{4\beta^{2}\max_{0\le s\le\gamma}|B_{s}-\widetilde{B}_{s}|\}-1)^{2})^{1/2}\le C\gamma^{2},
\end{align*}
where $C$ is a constant that may depend on $\beta$ and $R$. Since
$\mathbb{E}_{B}^{0}|B_{\gamma}|^{2}=\gamma d$, this proves the lemma.
\end{proof}
\begin{cor}
\label{cor:aSTregularize}We have
\[
a_{S,T}(0,0)=\lim_{\gamma\downarrow0}a_{S,T;\gamma},
\]
where
\begin{equation}
a_{S,T;\gamma}=1+\frac{2}{d\gamma}\widehat{\mathbb{E}}_{B;[-T,S]}^{0}(B_{S}-B_{\gamma})\cdot B_{\gamma}.\label{eq:aSTgdef}
\end{equation}
\end{cor}

\begin{proof}
This is a simple consequence of \lemref{aSTCM} and \lemref{smallscalesd}.
\end{proof}
\begin{lem}
\label{lem:limitexists}The limit
\begin{equation}
\lim_{\substack{T\to\infty\\
S\to\infty
}
}a_{S,T}(0,0)\label{eq:aSTlimit}
\end{equation}
exists.
\end{lem}

\begin{proof}
We have, for any $\tau_{1}<\tau_{2}<\tau_{3}<\tau_{4}\le\tau_{5}$,
\begin{align}
\widehat{\mathbb{E}}_{B;\tau_{5}}(B_{\tau_{4}}-B_{\tau_{3}})\cdot(B_{\tau_{2}}-B_{\tau_{1}}) & =\widehat{\mathbb{E}}_{W}(W_{\tau_{4}}-W_{\tau_{5}})\cdot(W_{\tau_{2}}-W_{\tau_{1}})\mathscr{G}(w_{\lfloor\tau_{5}\rfloor-1})\nonumber \\
 & =\widehat{\mathbb{E}}_{W}(W_{\tau_{4}\wedge\sigma}-W_{\tau_{3}\wedge\sigma})\cdot(W_{\tau_{2}}-W_{\tau_{1}}),\label{eq:incBincW}
\end{align}
where $\sigma$ is the first regeneration time after $\tau_{4}$ and
the second equality comes from the fact that $\mathscr{G}$ is even
and the increments of $W$ after a regeneration time are isotropic.
This makes it clear that there are constants $0<c,C<\infty$ so that
\begin{equation}
\widehat{\mathbb{E}}_{B;\tau_{5}}(B_{\tau_{4}}-B_{\tau_{3}})\cdot(B_{\tau_{2}}-B_{\tau_{1}})\le C\e^{-c(\tau_{3}-\tau_{2})},\label{eq:farawaydoesntmatter}
\end{equation}
since the increments of $W$ have exponential tails and, conditional
on there being a regeneration time in $(\tau_{2},\tau_{3})$, the
expectation of the right-hand side of \eqref{incBincW} is $0$. Then
it follows from \corref{aSTregularize} that $a_{S,T}$ is Cauchy
in $S$ and also in $T$, and thus the limit \eqref{aSTlimit} exists.
\end{proof}
Now we prove \thmref{aSTthm}.
\begin{proof}[Proof of \thmref{aSTthm}.]
We have, using \eqref{adef}, \eqref{farawaydoesntmatter}, and \lemref{cutofftail},
that
\begin{equation}
a=\lim_{U\to\infty}\frac{1}{dU}\widetilde{\mathbb{E}}_{W}(W_{3U}-W_{0})\cdot(W_{2U}-W_{U})=\lim_{U\to\infty}\frac{1}{dU}\widehat{\mathbb{E}}_{B;3U}^{0}(B_{3U}-B_{0})\cdot(B_{2U}-B_{U}).\label{eq:aU}
\end{equation}
Define
\[
\tau_{j}^{(\gamma)}=(U+j\gamma)\wedge2U
\]
and note that
\[
B_{2U}-B_{U}=\sum_{j=1}^{\lceil U/\gamma\rceil-1}(B_{\tau_{j+1}^{(\gamma)}}-B_{\tau_{j}^{(\gamma)}}).
\]
Substituting this into \eqref{aU} yields
\begin{multline*}
a=\lim_{U\to\infty}\frac{1}{dU}\lim_{\gamma\downarrow0}\widehat{\mathbb{E}}_{B;3U}^{0}(B_{3U}-B_{0})\cdot\sum_{j=0}^{\lceil U/\gamma\rceil-1}(B_{\tau_{j+1}^{(\gamma)}}-B_{\tau_{j}^{(\gamma)}})\\
=\lim_{U\to\infty}\frac{1}{dU}\lim_{\gamma\downarrow0}\sum_{j=0}^{\lceil U/\gamma\rceil-1}\widehat{\mathbb{E}}_{B;3U}^{0}\left((B_{3U}-B_{\tau_{j+1}^{(\gamma)}})+(B_{\tau_{j+1}^{(\gamma)}}-B_{\tau_{j}^{(\gamma)}})+(B_{\tau_{j}^{(\gamma)}}-B_{0})\right)\cdot(B_{\tau_{j+1}^{(\gamma)}}-B_{\tau_{j}^{(\gamma)}}).
\end{multline*}
Now by \lemref{smallscalesd}, we have
\[
\lim_{U\to\infty}\frac{1}{dU}\lim_{\gamma\downarrow0}\sum_{j=0}^{\lceil U/\gamma\rceil-1}\widehat{\mathbb{E}}_{B;3U}^{0}(B_{\tau_{j+1}^{(\gamma)}}-B_{\tau_{j}^{(\gamma)}})\cdot(B_{\tau_{j+1}^{(\gamma)}}-B_{\tau_{j}^{(\gamma)}})=1.
\]
Moreover, we have by \eqref{aSTgdef} that
\begin{align*}
\widehat{\mathbb{E}}_{B;3U}^{0}(B_{3U}-B_{\tau_{j+1}^{(\gamma)}})\cdot(B_{\tau_{j+1}^{(\gamma)}}-B_{\tau_{j}^{(\gamma)}}) & =\frac{\gamma d}{2}(a_{3U-\tau_{j}^{(\gamma)},\tau_{j}^{(\gamma)};\tau_{j+1}^{(\gamma)}-\tau_{j}^{(\gamma)}}-1);\\
\widehat{\mathbb{E}}_{B;3U}^{0}(B_{\tau_{j}^{(\gamma)}}-B_{0})\cdot(B_{\tau_{j+1}^{(\gamma)}}-B_{\tau_{j}^{(\gamma)}}) & =\frac{\gamma d}{2}(a_{\tau_{j+1}^{(\gamma)},3U-\tau_{j+1}^{(\gamma)};\tau_{j+1}^{(\gamma)}-\tau_{j}^{(\gamma)}}-1).
\end{align*}
Therefore,
\begin{align*}
a & =1+\lim_{U\to\infty}\frac{1}{dU}\lim_{\gamma\downarrow0}\sum_{j=0}^{\lceil U/\gamma\rceil-1}\left(\frac{\gamma d}{2}(a_{3U-\tau_{j}^{(\gamma)},\tau_{j}^{(\gamma)};\tau_{j+1}^{(\gamma)}-\tau_{j}^{(\gamma)}}-1)+\frac{\gamma d}{2}(a_{\tau_{j+1}^{(\gamma)},3U-\tau_{j+1}^{(\gamma)};\tau_{j+1}^{(\gamma)}-\tau_{j}^{(\gamma)}}-1)\right)\\
 & =\lim_{U\to\infty}\frac{1}{U}\lim_{\gamma\downarrow0}\frac{\gamma}{2}\sum_{j=0}^{\lceil U/\gamma\rceil-1}(a_{3U-\tau_{j}^{(\gamma)},\tau_{j}^{(\gamma)};\tau_{j+1}^{(\gamma)}-\tau_{j}^{(\gamma)}}+a_{\tau_{j+1}^{(\gamma)},3U-\tau_{j+1}^{(\gamma)};\tau_{j+1}^{(\gamma)}-\tau_{j}^{(\gamma)}})\\
 & =\lim_{\substack{T\to\infty\\
S\to\infty
}
}a_{S,T}(0,0),
\end{align*}
where the last equality is by \lemref{limitexists}.
\end{proof}

\section{Strong convergence of the leading term\label{sec:convergence}}

In this section, we prove \thmref{convergence}: convergence of the
leading term in the homogenization expansion \eqref{expansion}.
We begin by deriving an expression for the error in \eqref{convergence}
using the Feynman--Kac formula. We will use the Fourier transform
for the initial condition $u_{0}\in\mathcal{C}_{c}^{\infty}(\mathbb{R}^{d})$,
which we normalize as
\[
\widehat{u_{0}}(\omega)=\int\e^{-\i\omega\cdot x}u_{0}(x)\frac{\dif x}{(2\pi)^{d}},\qquad u_{0}(x)=\int\e^{\i\omega\cdot x}\widehat{u_{0}}(\omega)\,\dif\omega.
\]
In this section $\omega$ and $\tilde{\omega}$ denote Fourier variables;
the function $\boldsymbol{\omega}$ from the previous section makes
no appearance.
\begin{prop}
We have that
\begin{equation}
\mathbf{E}|u^{\eps}(t,x)-\Psi^{\eps}(t,x)\overline{u}(t,x)|^{2}=\e^{2\alpha_{\eps^{-2}t}}\int\int\e^{\i(\omega+\tilde{\omega})\cdot x}\widehat{u_{0}}(\omega)\widehat{u_{0}}(\tilde{\omega})\widehat{\mathbb{E}}_{B,\widetilde{B};\eps^{-2}t}\mathscr{A}_{t;\omega,\tilde{\omega}}^{\eps}[B,\widetilde{B}]\,\dif\omega\,\dif\tilde{\omega},\label{eq:errorexpr}
\end{equation}
where
\begin{align}
\mathscr{A}_{t;\omega,\tilde{\omega}}^{\eps}[B,\widetilde{B}] & =\exp\left\{ \beta^{2}\mathscr{R}_{\eps^{-2}t}[B,\widetilde{B}]\right\} \mathscr{E}_{t,\omega}^{\eps}[B]\mathscr{E}_{t,\tilde{\omega}}^{\eps}[\widetilde{B}];\label{eq:scradef}\\
\mathscr{E}_{t,\omega}^{\eps}[B] & =\e^{\i\omega\cdot\eps(B_{\eps^{-2}t}-B_{0})}-\e^{-\frac{1}{2}at|\omega|^{2}}.\label{eq:scredef}
\end{align}
\end{prop}

\begin{proof}
We start with the Feynman--Kac formula for \eqref{uepsPDE}--\eqref{uepsIC}:
\begin{equation}
u^{\eps}(t,x)=\mathbb{E}_{B}^{\eps^{-1}x}\exp\left\{ \beta\mathscr{V}_{\eps^{-2}t}[B]-\lambda\eps^{-2}t\right\} u_{0}(\eps B_{\eps^{-2}t}),\label{eq:uepsFK}
\end{equation}
and note that
\[
u_{0}(\eps B_{\eps^{-2}t})=\int\e^{\i\omega\cdot\eps B_{\eps^{-2}t}}\widehat{u_{0}}(\omega)\,\dif\omega,\qquad\overline{u}(t,x)=\int\e^{\i\omega\cdot x-\frac{1}{2}at|\omega|^{2}}\widehat{u_{0}}(\omega)\,\dif\omega,
\]
so, if $B_{0}=\eps^{-1}x$, then
\begin{equation}
u_{0}(\eps B_{\eps^{-2}t})-\overline{u}(t,x)=\int\e^{\i\omega\cdot x}\mathscr{E}_{t,\omega}^{\eps}[B]\widehat{u_{0}}(\omega)\,\dif\omega.\label{eq:uubarFT}
\end{equation}
The Feynman--Kac formula also shows that
\begin{equation}
\Psi^{\eps}(t,x)=\mathbb{E}_{B}^{\eps^{-1}x}\exp\left\{ \beta\mathscr{V}_{\eps^{-2}t}[B]-\lambda\eps^{-2}t\right\} .\label{eq:PsiepsFK}
\end{equation}
This is simply \eqref{uepsFK} with initial condition $u_{0}\equiv1$;
we also saw the unrescaled version before in \eqref{PsiFKpreliminary}.
Combining \eqref{uepsFK}, \eqref{uubarFT}, and \eqref{PsiepsFK}
yields
\[
u^{\eps}(t,x)-\Psi^{\eps}(t,x)\overline{u}(t,x)=\mathbb{E}_{B}^{\eps^{-1}x}\exp\left\{ \beta\mathscr{V}_{\eps^{-2}t}[B]-\lambda\eps^{-2}t\right\} \int\e^{\i\omega\cdot x}\mathscr{E}_{t,\omega}^{\eps}[B]\widehat{u_{0}}(\omega)\,\dif\omega.
\]
We finish the proof of the lemma by simply computing the second moment:
\begin{align*}
\mathbf{E} & (u^{\eps}(t,x)-\Psi^{\eps}(t,x)\overline{u}(t,x))^{2}=\mathbf{E}\left(\mathbb{E}_{B}^{\eps^{-1}x}\exp\left\{ \beta\mathscr{V}_{\eps^{-2}t}[B]-\lambda\eps^{-2}t\right\} \int\e^{\i\omega\cdot x}\mathscr{E}_{t,\omega}^{\eps}[B]\widehat{u_{0}}(\omega)\,\dif\omega\right)^{2}\\
 & =\int\int\e^{\i(\omega+\tilde{\omega})\cdot x}\widehat{u_{0}}(\omega)\widehat{u_{0}}(\tilde{\omega})\mathbb{E}_{B,\widetilde{B}}^{\eps^{-1}x,\eps^{-1}x}\mathbf{E}\exp\left\{ \mathscr{V}_{\eps^{-2}t}[B]+\mathscr{V}_{\eps^{-2}t}[\widetilde{B}]-2\lambda\eps^{-2}t\right\} \mathscr{E}_{t,\omega}^{\eps}[B]\mathscr{E}_{t,\tilde{\omega}}^{\eps}[\widetilde{B}]\,\dif\omega\,\dif\tilde{\omega}\\
 & =\e^{2\alpha_{\eps^{-2}t}}\int\int\e^{\i(\omega+\tilde{\omega})\cdot x}\widehat{u_{0}}(\omega)\widehat{u_{0}}(\omega)\widehat{\mathbb{E}}_{B,\widetilde{B};\eps^{-2}t}\mathscr{A}_{t;\omega,\tilde{\omega}}^{\eps}[B,\widetilde{B}]\,\dif\omega\,\dif\tilde{\omega}.\qedhere
\end{align*}
\end{proof}
To prove \thmref{convergence}, we will bound the expression on the
right-hand side of \eqref{errorexpr} using the techniques of \cite{GRZ17}
recalled in \secref{GRZreview}. On first reading, the reader may
again wish to consider the case when $V$ is white in time, so
the tilting of the Markov chain can be ignored and $B$ and $\widetilde{B}$
are simply Brownian motions. The key idea is that with high probability,
the only contributions to $\exp\left\{ \beta^{2}\mathscr{R}_{\eps^{-2}t}[B,\widetilde{B}]\right\} $
come from times close to $0$, so the expectation of \eqref{scradef}
``almost'' splits into a product of the expectations of $\mathscr{E}_{t,\omega}^{\eps}[B]$
and $\mathscr{E}_{t,\omega}^{\eps}[\widetilde{B}]$. Since the Markov
chain has effective diffusivity $a$, each of the latter expectations
is approximately $0$. (In the white-in-time case, $a=1$, and each
of the latter expectations is exactly $0$.)

Our first lemma is that the correction $\mathscr{G}$ appearing in
\eqref{Markovchainexpectation} does not matter.
\begin{lem}
\label{lem:enddoesntmatter}We have
\[
\lim_{\eps\to0}\left|\widehat{\mathbb{E}}_{B,\tilde{B};\eps^{-2}t}\mathscr{A}_{t;\omega,\tilde{\omega}}^{\eps}[B,\tilde{B}]-\widetilde{\mathbb{E}}_{W,\widetilde{W}}\mathscr{A}_{t;\omega,\tilde{\omega}}^{\eps}[W,\widetilde{W}]\right|=0.
\]
\end{lem}

As this lemma is a technical point, we defer its proof to the end
of this section. Now we note that, for $r,\tilde{r}\ge0$, we have
\begin{align}
\frac{\partial^{2}}{\partial r\partial\tilde{r}}\exp\left\{ \beta^{2}\mathscr{R}_{r,\tilde{r}}[W,\widetilde{W}]\right\}  & =\frac{\partial}{\partial r}\left[\left(\beta^{2}\int_{0}^{r}R(\tau-\tilde{r},W_{\tau}-\widetilde{W}_{\tilde{r}})\,\dif\tau\right)\exp\left\{ \beta^{2}\mathscr{R}_{r,\tilde{r}}[W,\widetilde{W}]\right\} \right]\nonumber \\
 & =\mathscr{Q}_{r,\tilde{r}}[W,\widetilde{W}]\exp\left\{ \beta^{2}\mathscr{R}_{r,\tilde{r}}[W,\widetilde{W}]\right\} ,\label{eq:partialderiv}
\end{align}
where
\begin{equation}
\mathscr{Q}_{r,\tilde{r}}[W,\widetilde{W}]=\beta^{2}R(r-\tilde{r},W_{r}-\widetilde{W}_{\tilde{r}})+\beta^{4}\int_{[\tilde{r}-2,r]}R(\tau-\tilde{r},W_{\tau}-\widetilde{W}_{\tilde{r}})\,\dif\tau\int_{[r-2,\tilde{r}]}R(r-\tilde{\tau},W_{r}-\widetilde{W}_{\tilde{\tau}})\,\dif\tilde{\tau}.\label{eq:Qdef}
\end{equation}
We note that, for each $r,\tilde{r}$,
\begin{equation}
\mathscr{Q}_{r,\tilde{r}}[W,\widetilde{W}]\ge0\label{eq:Qpositive}
\end{equation}
almost surely, since $R$ was assumed nonnegative. Now if we define
the shorthand
\[
\mathscr{E}_{t;\omega,\tilde{\omega}}^{\eps}[W,\widetilde{W}]=\mathscr{E}_{t,\omega}^{\eps}[W]\mathscr{E}_{t,\tilde{\omega}}^{\eps}[\widetilde{W}],
\]
then we can write
\begin{align}
\widetilde{\mathbb{E}}_{W,\widetilde{W}}\mathscr{A}_{t;\omega,\tilde{\omega}}^{\eps}[W,\widetilde{W}] & =\widetilde{\mathbb{E}}_{W,\widetilde{W}}\mathscr{E}_{t;\omega,\tilde{\omega}}^{\eps}[W,\widetilde{W}]\exp\left\{ \beta^{2}\mathscr{R}_{\eps^{-2}t}[W,\widetilde{W}]\right\} \nonumber \\
 & =\int_{0}^{\eps^{-2}t}\int_{0}^{\eps^{-2}t}\widetilde{\mathbb{E}}_{W,\widetilde{W}}\mathscr{E}_{t;\omega,\tilde{\omega}}^{\eps}[W,\widetilde{W}]\mathscr{Q}_{r,\tilde{r}}[W,\widetilde{W}]\exp\left\{ \beta^{2}\mathscr{R}_{r,\tilde{r}}[W,\widetilde{W}]\right\} \,\dif r\,\dif\tilde{r}.\label{eq:doubleintegral}
\end{align}

The next lemma gives an estimate for the contribution to the integral
\eqref{doubleintegral} from each $r,\tilde{r}$. The key point is
that, if $B$ is a Brownian motion with diffusivity $\sigma^{2}$,
then $\exp\left\{ \i\omega\cdot B_{t}+\frac{1}{2}t\sigma^{2}|\omega|^{2}\right\} $
is a martingale. Since $W$ is converging to a Brownian motion with
diffusivity $a$, the contribution to the integrand in \eqref{doubleintegral}
from $\mathscr{E}_{t;\omega,\tilde{\omega}}^{\eps}[W,\widetilde{W}]$
should be small except for the contribution from time interval $[0,r\vee\tilde{r}]$,
on which the term $\exp\left\{ \beta^{2}\mathscr{R}_{r,\tilde{r}}[W,\widetilde{W}]\right\} $
could have an effect. But for fixed $r,\tilde{r}$, this time interval
is microscopic, and thus does not contribute in the limit.

\begin{lem}
\label{lem:ptwiseconvergence}For fixed $r,\tilde{r}\ge0$, we have
\[
\lim_{\eps\to0}\widetilde{\mathbb{E}}_{W,\widetilde{W}}\mathscr{E}_{t;\omega,\tilde{\omega}}^{\eps}[W,\widetilde{W}]\mathscr{Q}_{r,\tilde{r}}[W,\widetilde{W}]\exp\left\{ \beta^{2}\mathscr{R}_{r,\tilde{r}}[W,\widetilde{W}]\right\} =0.
\]
\end{lem}

\begin{proof}
In this proof we will treat $r$ and $\tilde{r}$ as fixed, and suppress
them from the notation of the objects we define. We abbreviate $\sigma_{j}=\sigma_{j}^{W,\widetilde{W}}$
from \eqref{sigmanWWdef} and recall the definition \eqref{kappadef}
of $\kappa_{2}$. Let $j_{0}\in\{j\ge0\mid\sigma_{j}\ge r\vee\tilde{r}\}$
and let the $\sigma$-algebra $\mathcal{F}_{j_{0}}$ be generated
by the collection of random variables
\[
\{\eta_{n}^{W,\tilde{W}}\mid n<\sigma_{j_{0}}\}\cup\{w_{n}\mid n<\sigma_{j_{0}}\}\cup\{\tilde{w}_{n}\mid n<\sigma_{j_{0}}\},
\]
with notation as in \thmref{BstoWs}. We note that the random variable
\[
\mathscr{Q}_{r,\tilde{r}}[W,\widetilde{W}]\exp\left\{ \beta^{2}\mathscr{R}_{r,\tilde{r}}[W,\widetilde{W}]\right\} 
\]
is $\mathcal{F}_{j_{0}}$-measurable. Therefore, we have
\begin{align}
\widetilde{\mathbb{E}}_{W,\widetilde{W}} & \mathscr{E}_{t;\omega,\tilde{\omega}}^{\eps}[W,\widetilde{W}]\mathscr{Q}_{r,\tilde{r}}[W,\widetilde{W}]\exp\left\{ \beta^{2}\mathscr{R}_{r,\tilde{r}}[W,\widetilde{W}]\right\} \nonumber \\
 & =\widetilde{\mathbb{E}}_{W,\widetilde{W}}\left(\widetilde{\mathbb{E}}_{W,\widetilde{W}}\left[\mathscr{E}_{t;\omega,\tilde{\omega}}^{\eps}[W,\widetilde{W}]\;\middle|\;\mathcal{F}_{j_{0}}\right]\exp\left\{ \beta^{2}\mathscr{R}_{r,\tilde{r}}[W,\widetilde{W}]\right\} \mathscr{Q}_{r,\tilde{r}}[W,\widetilde{W}]\right)\nonumber \\
 & =\widetilde{\mathbb{E}}_{W,\widetilde{W}}\left(\e^{\i\omega\cdot\eps W_{j_{0}}}\widetilde{\mathbb{E}}_{W}\left[\e^{\i\omega\cdot\eps(W_{\eps^{-2}t}-W_{j_{0}})}\;\middle|\;\mathcal{F}_{j_{0}}\right]-\e^{-\frac{1}{2}at|\omega|^{2}}\right)\cdot\nonumber \\
 & \qquad\qquad\cdot\left(\e^{\i\tilde{\omega}\cdot\eps\widetilde{W}_{j_{0}}}\widetilde{\mathbb{E}}_{W}\left[\e^{\i\tilde{\omega}\cdot\eps(\widetilde{W}_{\eps^{-2}t}-\widetilde{W}_{j_{0}})}\;\middle|\;\mathcal{F}_{j_{0}}\right]-\e^{-\frac{1}{2}at|\tilde{\omega}|^{2}}\right)\mathscr{Q}_{r,\tilde{r}}[W,\widetilde{W}]\exp\left\{ \beta^{2}\mathscr{R}_{r,\tilde{r}}[W,\widetilde{W}]\right\} .\label{eq:conditioning}
\end{align}
Observe that
\[
\widetilde{\mathbb{E}}_{W}\left[\e^{\i\omega\cdot\eps(W_{\eps^{-2}t}-W_{j_{0}})}\;\middle|\;\mathcal{F}_{j_{0}}\right]=\widetilde{\mathbb{E}}_{W}\left[\e^{\i\omega\cdot\eps(W_{\eps^{-2}t}-W_{j_{0}})}\;\middle|\;j_{0}\right]\to\e^{-\frac{1}{2}at|\omega|^{2}}
\]
almost surely as $\eps\to0$ by \propref{invariance}, and similarly
for $\widetilde{\mathbb{E}}_{\widetilde{W}}\left[\e^{\i\tilde{\omega}\cdot\eps(\widetilde{W}_{\eps^{-2}t}-\widetilde{W}_{j_{0}})}\;\middle|\;\mathcal{F}_{j_{0}}\right]$.
In addition, we have
\[
\e^{\i\tilde{\omega}\cdot\eps\widetilde{W}_{j_{0}}}\to1
\]
almost surely as $\eps\to0$. The statement of the lemma then follows
from the bounded convergence theorem applied to \eqref{conditioning}.
\end{proof}
Now we upgrade the pointwise convergence to convergence of the integral.
\begin{lem}
\label{lem:AWWtozero}We have
\[
\lim_{\eps\to0}\widetilde{\mathbb{E}}_{W,\widetilde{W}}\mathscr{A}_{t;\omega,\tilde{\omega}}^{\eps}[W,\widetilde{W}]=0.
\]
\end{lem}

\begin{proof}
Using \eqref{Qpositive}, we have
\[
\left|\widetilde{\mathbb{E}}_{W,\widetilde{W}}\mathscr{E}_{t,\omega,\tilde{\omega}}^{\eps}[W,\widetilde{W}]\mathscr{Q}_{r,\tilde{r}}[W,\widetilde{W}]\exp\left\{ \beta^{2}\mathscr{R}_{r,\tilde{r}}[W,\widetilde{W}]\right\} \right|\le4\widetilde{\mathbb{E}}_{W,\widetilde{W}}\mathscr{Q}_{r,\tilde{r}}[W,\widetilde{W}]\exp\left\{ \beta^{2}\mathscr{R}_{r,\tilde{r}}[W,\widetilde{W}]\right\} .
\]
Using \eqref{partialderiv}, we have that
\begin{align*}
\int_{0}^{\tilde{q}} & \int_{0}^{q}\widetilde{\mathbb{E}}_{W,\widetilde{W}}\mathscr{Q}_{r,\tilde{r}}[W,\widetilde{W}]\exp\left\{ \beta^{2}\mathscr{R}_{r,\tilde{r}}[W,\widetilde{W}]\right\} \,\dif r\,\dif\tilde{r}\\
 & =\widetilde{\mathbb{E}}_{W,\widetilde{W}}\exp\left\{ \beta^{2}\mathscr{R}_{q,\tilde{q}}[W,\widetilde{W}]\right\} \le\widetilde{\mathbb{E}}_{W,\widetilde{W}}\exp\left\{ \beta^{2}\mathscr{R}_{\infty}[W,\widetilde{W}]\right\} <\infty,
\end{align*}
where the last equality is by \propref{expexpbdd}. The dominated
convergence theorem applied to the integral \eqref{doubleintegral},
in light of the pointwise convergence established in \lemref{ptwiseconvergence},
then implies the result.
\end{proof}
We are now ready to prove \thmref{convergence}.
\begin{proof}[Proof of \thmref{convergence}.]
Combining \lemref[s]{enddoesntmatter}~and~\lemref{AWWtozero},
we see that the integrand in \eqref{errorexpr} converges pointwise
to $0$ as $\eps\to0$. On the other hand, by \propref{expexpbdd},
as long as $\beta<\beta_{0}$, there is a constant $C$ so that
\[
\left|\widehat{\mathbb{E}}_{B,\tilde{B};\eps^{-2}t}\mathscr{A}_{t;\omega,\tilde{\omega}}^{\eps}[B,\widetilde{B}]\right|\le C
\]
independently of $\eps,\omega,\tilde{\omega}$. As $u_{0}\in\mathcal{C}_{\mathrm{c}}^{\infty}(\mathbb{R}^{d})$,
the dominated convergence theorem and \eqref{alphaconverges} imply
that
\[
\mathbf{E}|u^{\eps}(t,x)-\Psi^{\eps}(t,x)\overline{u}(t,x)|^{2}\to0
\]
as $\eps\to0$.
\end{proof}
It remains to prove \lemref{enddoesntmatter}.
\begin{proof}[Proof of \lemref{enddoesntmatter}.]
We have
\begin{align}
\widehat{\mathbb{E}}_{B,\widetilde{B};\eps^{-2}t}\mathscr{A}_{t;\omega,\tilde{\omega}}^{\eps}[B,\widetilde{B}] & =\widehat{\mathbb{E}}_{B,\widetilde{B};\eps^{-2}t}\mathscr{E}_{t;\omega,\tilde{\omega}}^{\eps}[B,\widetilde{B}]\exp\left\{ \beta^{2}\mathscr{R}_{\eps^{-2}t}[B,\widetilde{B}]\right\} \nonumber \\
 & =\widetilde{\mathbb{E}}_{W,\widetilde{W}}\mathscr{G}[w_{\lfloor\eps^{-2}t\rfloor-1}]\mathscr{G}[\tilde{w}_{\lfloor\eps^{-2}t\rfloor-1}]\mathscr{E}_{t;\omega,\tilde{\omega}}^{\eps}[W,\widetilde{W}]\exp\left\{ \beta^{2}\mathscr{R}_{\eps^{-2}t}[W,\widetilde{W}]\right\} .\label{eq:movetomc}
\end{align}
Let $\gamma\in(0,2)$ be arbitrary. Then
\begin{align}
\widetilde{\mathbb{E}}_{W,\widetilde{W}} & \left|\mathscr{E}_{t;\omega,\tilde{\omega}}^{\eps}[W,\widetilde{W}]\exp\left\{ \beta^{2}\mathscr{R}_{\eps^{-2}t}[W,\widetilde{W}]\right\} -\mathscr{E}_{t-\eps^{\gamma};\omega,\tilde{\omega}}^{\eps}[W,\widetilde{W}]\exp\left\{ \beta^{2}\mathscr{R}_{\eps^{-2}(t-\eps^{\gamma})}[W,\widetilde{W}]\right\} \right|\nonumber \\
 & \le\widetilde{\mathbb{E}}_{W,\widetilde{W}}\left|\mathscr{E}_{t;\omega,\tilde{\omega}}^{\eps}[W,\widetilde{W}]-\mathscr{E}_{t-\eps^{\gamma};\omega,\tilde{\omega}}^{\eps}[W,\widetilde{W}]\right|\exp\left\{ \beta^{2}\mathscr{R}_{\eps^{-2}t}[W,\widetilde{W}]\right\} \nonumber \\
 & \qquad+\widetilde{\mathbb{E}}_{W,\widetilde{W}}\left|\mathscr{E}_{t-\eps^{\gamma};\omega,\tilde{\omega}}^{\eps}[W,\widetilde{W}]\right|\left|\exp\left\{ \beta^{2}\mathscr{R}_{\eps^{-2}t}[W,\widetilde{W}]\right\} -\exp\left\{ \beta^{2}\mathscr{R}_{\eps^{-2}(t-\eps^{\gamma})}\right\} \right|.\label{eq:breakdiff}
\end{align}
We begin by addressing the first term of \eqref{breakdiff}. By \eqref{West},
we have
\[
\widetilde{\mathbb{E}}_{W}\left|\eps W_{\eps^{-2}t}-\eps W_{\eps^{-2}(t-\eps^{\gamma})}\right|^{2}\le C\eps^{\gamma},
\]
which in particular means that
\begin{equation}
\lim_{\eps\to0}\left|\eps W_{\eps^{-2}t}-\eps W_{\eps^{-2}(t-\eps^{\gamma})}\right|=0\label{eq:Wdiffconvergesinprob}
\end{equation}
in probability. The same statement of course holds for $\widetilde{W}$.
We then have, using Hölder's inequality, that for $\delta>0$ sufficiently
small there is a constant $C_{\delta}$ so that
\begin{equation}\label{eq:firstlimit}
\begin{split}
\lim_{\eps\to0}\widetilde{\mathbb{E}}_{W,\widetilde{W}}  \left|\mathscr{E}_{t;\omega,\tilde{\omega}}^{\eps}[W,\widetilde{W}]-\mathscr{E}_{t-\eps^{\gamma};\omega,\tilde{\omega}}^{\eps}[W,\widetilde{W}]\right|\exp\left\{ \beta^{2}\mathscr{R}_{\eps^{-2}t}[W,\widetilde{W}]\right\}\\
  \le C_{\delta}\lim_{\eps\to0}\left(\widetilde{\mathbb{E}}_{W,\widetilde{W}}\left|\mathscr{E}_{t;\omega,\tilde{\omega}}^{\eps}[W,\widetilde{W}]-\mathscr{E}_{t-\eps^{\gamma};\omega,\tilde{\omega}}^{\eps}[W,\widetilde{W}]\right|^{1/\delta+1}\right)=0
\end{split}
\end{equation}
by \propref{expexpbdd} and the bounded convergence theorem in light
of \eqref{Wdiffconvergesinprob}.

Finally, we consider the second term of \eqref{breakdiff}, which
is easier. Here, we have
\begin{equation}\label{eq:secondlimit}
\begin{split}
\lim_{\eps\to0}\widetilde{\mathbb{E}}_{W,\widetilde{W}}\left|\mathscr{E}_{t-\eps^{\gamma};\omega,\tilde{\omega}}^{\eps}[W,\widetilde{W}]\right|  \left|\exp\left\{ \beta^{2}\mathscr{R}_{\eps^{-2}t}[W,\widetilde{W}]\right\} -\exp\left\{ \beta^{2}\mathscr{R}_{\eps^{-2}(t-\eps^{\gamma})}[W,\widetilde{W}]\right\} \right| \\
  \le4\lim_{\eps\to0}\widetilde{\mathbb{E}}_{W,\widetilde{W}}\left|\exp\left\{ \beta^{2}\mathscr{R}_{\eps^{-2}t}[W,\widetilde{W}]\right\} -\exp\left\{ \beta^{2}\mathscr{R}_{\eps^{-2}(t-\eps^{\gamma})}[W,\widetilde{W}]\right\} \right|=0
\end{split}
\end{equation}
by the dominated convergence theorem, again in light of \eqref{West}.
Applying \eqref{firstlimit} and \eqref{secondlimit} to \eqref{breakdiff}
implies that
\begin{equation}
\lim_{\eps\to0}\widehat{\mathbb{E}}_{W,\widetilde{W}}\left|\mathscr{E}_{t;\omega,\tilde{\omega}}^{\eps}[W,\widetilde{W}]\exp\left\{ \beta^{2}\mathscr{R}_{\eps^{-2}t}[W,\widetilde{W}]\right\} -\mathscr{E}_{t-\eps^{\gamma};\omega,\tilde{\omega}}^{\eps}[W,\widetilde{W}]\exp\left\{ \beta^{2}\mathscr{R}_{\eps^{-2}(t-\eps^{\gamma})}[W,\widetilde{W}]\right\} \right|=0.\label{eq:mcenddoesntmatter}
\end{equation}
Combining \eqref{movetomc}, \eqref{mcenddoesntmatter}, and \lemref{cutofftail},
and recalling that $\mathscr{G}$ is bounded from above and away from
zero, completes the proof of the lemma.
\end{proof}

\section{The second term of the expansion\label{sec:weakconvergence}}

In this section we will prove \thmref{weakconvergence}. We first
introduce some notation. Fix $\gamma\in(1,2)$ and $t>0$; all constants
in this section will depend on $\gamma$ and $t$. We define a discrete
set of times
\begin{equation}
r_{k}=\begin{cases}
0 & k=0;\\
t-\eps^{-\gamma}(\lfloor\eps^{\gamma}t\rfloor-(k-1)) & k>0,
\end{cases}\label{eq:rkdef}
\end{equation}
and set
\begin{equation}
\mathscr{I}_{t}^{\eps}[B]=\sum_{k=0}^{K_{t}^{\eps}}(\eps B_{r_{k+1}}-\eps B_{r_{k}})\cdot\nabla\overline{u}(t-\eps^{2}r_{k},\eps B_{r_{k}}),\label{eq:Idef}
\end{equation}
with
\begin{equation}
K_{t}^{\eps}=\lfloor\eps^{\gamma-2}t\rfloor.\label{eq:Kdef}
\end{equation}
The next lemma gives a Feynman--Kac formula for the corrector $u_{1}^{\eps}$
defined in \eqref{u1epsdef}.
\begin{lem}
We have
\begin{equation}
u_{1}^{\eps}(t,x)=\frac{1}{\eps}\mathbb{E}_{B}^{\eps^{-1}x}\exp\left\{ \mathscr{V}_{\eps^{-2}t}[B]-\lambda\eps^{-2}t\right\} \mathscr{I}_{t}^{\eps}[B].\label{eq:u1epsfk}
\end{equation}
\end{lem}

\begin{proof}
The Feynman--Kac formula applied to \eqref{thetajproblem}, in the
same way as \eqref{omegaSFKfirstcomp}, gives the following expression
for the solution $\theta_{j}(s,y)$ to that equation:
\begin{align*}
\theta_{j}(s,y) & =\mathbb{E}_{B}^{y}\int_{0}^{s-\eps^{-\gamma}(j-1)}\exp\left\{ \int_{0}^{r}[\beta V(s-\tau,B_{\tau})-\lambda]\,\dif\tau\right\} \nabla\Psi(s-r,B_{r})\,\dif r\\
 & =\mathbb{E}_{B}^{y}\nabla_{\xi}|_{\xi=0}\int_{0}^{s-\eps^{-\gamma}(j-1)}\exp\left\{ \int_{0}^{s}[\beta V(s-\tau,B_{\tau}+H(\tau-r)\xi)-\lambda]\,\dif\tau\right\} \,\dif r\\
 & =\nabla_{\xi}|_{\xi=0}\mathbb{E}_{B}^{y}\exp\left\{ \int_{0}^{s}[\beta V(s-\tau,B_{\tau}+(\tau\wedge(s-\eps^{-\gamma}(j-1)))\xi)-\lambda]\,\dif\tau\right\} ,
\end{align*}
where $H$ is the Heaviside function. The Girsanov formula then yields
\begin{align*}
\theta_{j}(s,y) & =\nabla_{\xi}|_{\xi=0}\mathbb{E}_{B}^{y}\exp\left\{ \mathscr{V}_{s}[B]-\lambda s+(B_{s-\eps^{-\gamma}(j-1)}-y)\cdot\xi-\frac{s-\eps^{-\gamma}(j-1)}{2}|\xi|^{2}\right\} \\
 & =\mathbb{E}_{B}^{y}(B_{s-\eps^{-\gamma}(j-1)}-y)\exp\left\{ \mathscr{V}_{s}[B]-\lambda s\right\} .
\end{align*}
Given this expression for $\theta_{j}$, we can then write the Feynman--Kac
formula for \eqref{u1jproblem}:
\begin{align*}
u_{1;j}(s,y) & =\mathbb{E}_{B}^{y}\exp\left\{ \int_{0}^{s-\eps^{-\gamma}j}[\beta V(s-\tau,B_{\tau})-\lambda]\,\dif\tau\right\} \theta_{j}(\eps^{-\gamma}j,B_{s-\eps^{-\gamma}j})\cdot\nabla\overline{u}(\eps^{2-\gamma}j,\eps B_{s-\eps^{-\gamma}j})\\
 & =\mathbb{E}_{B}^{y}(B_{s-\eps^{-\gamma}(j-1)}-B_{s-\eps^{-\gamma}j})\cdot\nabla\overline{u}(\eps^{2-\gamma}j,\eps B_{s-\eps^{-\gamma}j})\exp\{\mathscr{V}_{s}[B]-\lambda s\}.
\end{align*}
Finally, by \eqref{u1epsdef} we have
\[
u_{1}^{\eps}(t,x)=u_{1}(\eps^{-2}t,\eps^{-1}x),
\]
where
\begin{align*}
u_{1}(s,y) & =\sum_{j=1}^{\lfloor\eps^{\gamma}s\rfloor}\mathbb{E}_{B}^{y}(B_{s-\eps^{-\gamma}(j-1)}-B_{s-\eps^{-\gamma}j})\cdot\nabla\overline{u}(\eps^{2-\gamma}j,\eps B_{s-\eps^{-\gamma}j})\exp\{\mathscr{V}_{s}[B]-\lambda s\}\\
 & \qquad+\mathbb{E}_{B}^{y}(B_{s-\eps^{-\gamma}\lfloor\eps^{\gamma}s\rfloor}-y)\exp\{\mathscr{V}_{s}[B]-\lambda s\}\cdot\nabla\overline{u}(\eps^{2}s,\eps y)\\
 & =\mathbb{E}_{B}^{y}\exp\left\{ \mathscr{V}_{s}[B]-\lambda s\right\} \sum_{k=0}^{\lfloor\eps^{\gamma}s\rfloor}(B_{r_{k+1}}-B_{r_{k}})\cdot\nabla\overline{u}(\eps^{2}(s-r_{k}),\eps B_{r_{k}}),
\end{align*}
with $r_{k}$ defined in \eqref{rkdef}; this yields \eqref{u1epsfk}.
\end{proof}
Next we consider the error term
\[
q^{\eps}(t,x)=u^{\eps}(t,x)-\Psi^{\eps}(t,x)\overline{u}(t,x)-\eps u_{1}^{\eps}(t,x).
\]
Combining \eqref{uepsFK}, \eqref{PsiepsFK}, and \eqref{u1epsfk}
gives the expression
\[
q^{\eps}(t,x)=\mathbb{E}_{B}^{\eps^{-1}x}\left[u_{0}(\eps B_{\eps^{-2}t})-\overline{u}(t,x)-\mathscr{I}_{t}^{\eps}[B]\right]\exp\left\{ \mathscr{V}_{\eps^{-2}t}[B]-\lambda\eps^{-2}t\right\} ,
\]
with expectation
\[
\mathbf{E}q^{\eps}(t,x)=\e^{\alpha_{\eps^{-2}t}}\widehat{\mathbb{E}}_{B;\eps^{-2}t}^{\eps^{-1}x}[u_{0}(\eps B_{\eps^{-2}t})-\overline{u}(t,x)-\mathscr{I}_{t}^{\eps}[B]].
\]
Taking covariances, we obtain
\begin{align}
\mathbf{E} & q^{\eps}(t,x)q^{\eps}(t,\tilde{x})-\mathbf{E}q^{\eps}(t,x)\mathbf{E}q^{\eps}(t,\tilde{x})\nonumber \\
 & =\e^{2\alpha_{\eps^{-2}t}}\widehat{\mathbb{E}}_{B,\widetilde{B};\eps^{-2}t}^{\eps^{-1}x,\eps^{-1}x}(u_{0}(\eps B_{\eps^{-2}t})-\overline{u}(t,x)-\mathscr{I}_{t}^{\eps}[B])(u_{0}(\eps\widetilde{B}_{\eps^{-2}t})-\overline{u}(t,\tilde{x})-\mathscr{I}_{t}^{\eps}[\widetilde{B}])\cdot\nonumber \\
 & \qquad\qquad\qquad\qquad\qquad\cdot\left(\exp\left\{ \beta^{2}\mathscr{R}_{\eps^{-2}t}[B,\widetilde{B}]\right\} -1\right)\nonumber \\
 & =\e^{2\alpha_{\eps^{-2}t}}\widetilde{\mathbb{E}}_{W,\widetilde{W}}^{\eps^{-1}x,\eps^{-1}\tilde{x}}(u_{0}(\eps W_{\eps^{-2}t})-\overline{u}(t,x)-\mathscr{I}_{t}^{\eps}[W])(u_{0}(\eps\widetilde{W}_{\eps^{-2}t})-\overline{u}(t,x)-\mathscr{I}_{t}^{\eps}[\widetilde{W}])\cdot\nonumber \\
 & \qquad\qquad\qquad\qquad\qquad\qquad\cdot\left(\exp\left\{ \beta^{2}\mathscr{R}_{\eps^{-2}t}[W,\widetilde{W}]\right\} -1\right)\mathscr{G}[w_{\lfloor\eps^{-2}t\rfloor-1}]\mathscr{G}[\tilde{w}_{\lfloor\eps^{-2}t\rfloor-1}].\label{eq:qcovcomp}
\end{align}
In the last equality of \eqref{qcovcomp} we used \thmref{BstoWs}.

In line with the framework of \secref{GRZreview}, we will proceed
to approximate the times $r_{k}$ by nearby regeneration times of
the Markov chain. Thus, we define
\begin{equation}
\sigma^{W}(k)=(\eps^{-2}t)\wedge\min\{r\ge r_{k}\mid\eta_{r}^{W}=1\},\label{eq:sigmaWkdef}
\end{equation}
where $\eta_{r}^{W}$ is as in \thmref{BstoWs}. Before we begin our
argument in earnest, we record bounds on the relevant error terms.
Put
\[
Y=\max_{0\le k\le K_{t}^{\eps}}(\sigma^{W}(k)-r_{k}),\qquad F(\tau)=\max_{r\in[0,\eps^{-2}t-\tau]}|W_{r+\tau}-W_{r}|,\qquad Z=\eps^{\gamma/2}F(\eps^{-\gamma}+Y).
\]

\begin{lem}
\label{lem:errorterms}We have constants $0<c,C<\infty$ so that,
for all $\xi\ge0$, we have
\begin{gather}
\widetilde{\mathbb{P}}_{W}(Y\ge C|\log\eps|+\xi)\le C\e^{-c\xi},\label{eq:Yexptail}\\
\widetilde{\mathbb{P}}_{W}(F(Y)\ge C|\log\eps|+\xi)\le C\e^{-c\xi},\label{eq:FYexptail}
\end{gather}
and
\begin{equation}
\widetilde{\mathbb{P}}_{W}(Z\ge C|\log\eps|+\xi)\le C\e^{-c\xi}.\label{eq:Zexptail}
\end{equation}
\end{lem}

These bounds are simple consequences of the regeneration structure
of the Markov chain described in \secref{GRZreview} and of \cite[Lemma A.1]{GRZ17}.
We begin our approximation procedure by replacing the deterministic
times $r_{k}$ in the definition \eqref{Idef} of $\mathscr{I}_{t,x}^{\eps}$
by the regeneration time approximations.
\begin{lem}
Let
\begin{equation}
\widetilde{\mathscr{I}}_{t}^{\eps}[W]=\sum_{k=0}^{K_{t}^{\eps}}(\eps W_{\sigma^{W}(k+1)}-\eps W_{\sigma^{W}(k)})\cdot\nabla\overline{u}(t-\eps^{2}\sigma^{W}(k),\eps W_{\sigma^{W}(k)}).\label{eq:Itildedef}
\end{equation}
For any $1\le p<\infty$ and any $\zeta<\gamma-1$ there exists a
constant $C=C(p,\zeta,t,\|u_{0}\|_{\mathcal{C}^{2}(\mathbb{R}^{d})})<\infty$
so that
\begin{equation}
\left(\widetilde{\mathbb{E}}_{W}^{x}|\mathscr{I}_{t}^{\eps}[W]-\widetilde{\mathscr{I}}_{t}^{\eps}[W]|^{p}\right)^{1/p}\le C\eps^{\zeta}.\label{eq:Idiffbd}
\end{equation}
\end{lem}

\begin{proof}
We have
\begin{align*}
 & \mathscr{I}_{t}^{\eps}[W]-\widetilde{\mathscr{I}}_{t}^{\eps}[W]=\sum_{k=0}^{K_{t}^{\eps}}\left[\vphantom{\sigma^{W}(k),\eps W_{\sigma^{W}(k)}}(\eps W_{r_{k+1}}-\eps W_{r_{k}})\cdot\nabla\overline{u}(t-\eps^{2}r_{k},\eps W_{r_{k}})\right.\\
 & \qquad\qquad\qquad\qquad\qquad\qquad\left.-(\eps W_{\sigma^{W}(k+1)}-\eps W_{\sigma^{W}(k)})\cdot\nabla\overline{u}(t-\eps^{2}\sigma^{W}(k),\eps W_{\sigma^{W}(k)})\right],
\end{align*}
hence
\begin{multline}
|\mathscr{I}_{t}^{\eps}[W]-\widetilde{\mathscr{I}}_{t}^{\eps}[W]|\le\sum_{k=0}^{K_{t}^{\eps}}|(\eps W_{r_{k}+1}-\eps W_{r_{k}})-(\eps W_{\sigma^{W}(k+1)}-\eps W_{\sigma^{W}(k)})|\cdot|\nabla\overline{u}(t-\eps^{2}r_{k},\eps W_{r_{k}})|\\
+\sum_{k=0}^{K_{t}^{\eps}}|\eps W_{\sigma^{W}(k+1)}-\eps W_{\sigma^{W}(k)}|\cdot|\nabla\overline{u}(t-\eps^{2}r_{k},\eps W_{r_{k}})-\nabla\overline{u}(t-\eps^{2}\sigma^{W}(k),\eps W_{\sigma^{W}(k)})|.\label{eq:Idiff}
\end{multline}
We bound above the first term on the right-hand side by 
\begin{equation}
|(\eps W_{r_{k+1}}-\eps W_{r_{k}})-(\eps W_{\sigma^{W}(k+1)}-\eps W_{\sigma^{W}(k)})|\cdot|\nabla\overline{u}(t-\eps^{2}r_{k})|\le2F(Y)\eps\|\overline{u}\|_{\mathcal{C}^{1}(\mathbb{R}^{d})},\label{eq:Idiff1}
\end{equation}
and the second by
\begin{multline}
|\eps W_{\sigma^{W}(k+1)}-\eps W_{\sigma^{W}(k)}|\cdot|\nabla\overline{u}(t-\eps^{2}r_{k},\eps W_{r_{k}})-\nabla\overline{u}(t-\eps^{2}\sigma^{W}(k),\eps W_{\sigma^{W}(k)})|\\
\le\eps F(Y+\eps^{-\gamma})\|\overline{u}\|_{\mathcal{C}^{2}(\mathbb{R}^{d})}(\eps^{2}Y+\eps F(Y))=\eps^{1-\gamma/2}Z\|\overline{u}\|_{\mathcal{C}^{2}(\mathbb{R}^{d})}(\eps^{2}Y+\eps F(Y)).\label{eq:Idiff2}
\end{multline}
Combining \eqref{Idiff}, \eqref{Idiff1}, and \eqref{Idiff2}, and
recalling the definition \eqref{Kdef} of $K_{t}^{\eps}$, gives us
\[
|\mathscr{I}_{t}^{\eps}[W]-\widetilde{\mathscr{I}}_{t}^{\eps}[W]|\le\eps^{\gamma-2}t\left[2F(Y)\eps\|\overline{u}\|_{\mathcal{C}^{1}(\mathbb{R}^{d})}+\eps^{1-\gamma/2}Z\|\overline{u}\|_{\mathcal{C}^{2}(\mathbb{R}^{d})}(\eps^{2}Y+\eps F(Y))\right],
\]
which in light of \lemref{errorterms} implies \eqref{Idiffbd}.
\end{proof}
\begin{lem}
For any power $1\le p<\infty$, there exists a $C=C(p,t,\zeta,\|u_{0}\|_{\mathcal{C}^{3}(\mathbb{R}^{d})})<\infty$
so that
\begin{equation}
\left(\widetilde{\mathbb{E}}_{W}^{\eps^{-1}x}\left|u_{0}(\eps W_{\eps^{-2}t})-\overline{u}(t,x)-\widetilde{\mathscr{I}}_{t}^{\eps}[W]\right|^{p}\right)^{1/p}\le C\eps^{\zeta}\label{eq:boundwithItilde}
\end{equation}
for any $\zeta<1-\gamma/2$.
\end{lem}

\begin{proof}
To ease the notation, in this proof we will abbreviate $\sigma=\sigma^{W}$.
(Recall the definition \eqref{sigmaWkdef}.) We write the Taylor expansion
\begin{equation}
\begin{aligned}\overline{u}( & t-\eps^{2}\sigma(k+1),\eps W_{\sigma(k+1)})-\overline{u}(t-\eps^{2}\sigma(k),\eps W_{\sigma(k)})\\
 & =-\eps^{2}(\sigma(k+1)-\sigma(k))\partial_{t}\overline{u}(t-\eps^{2}\sigma(k),\eps W_{\sigma(k)})+\eps(W_{\sigma(k+1)}-W_{\sigma(k)})\cdot\nabla\overline{u}(t-\eps^{2}\sigma(k),\eps W_{\sigma(k)})\\
 & \qquad+\frac{1}{2}\eps^2\mathrm{Q}\overline{u}(t-\eps^{2}\sigma(k),\eps W_{\sigma(k)})(W_{\sigma(k+1)}-W_{\sigma(k)})+\mathscr{Y}_{k}[W],
\end{aligned}
\label{eq:ubartaylor}
\end{equation}
where $\mathrm{Q}\overline{u}(t,x)$ is the quadratic form associated
to the Hessian of $\overline{u}$ at $(t,x)$ (so $\mathrm{Q}\overline{u}(t,x)(V)=\Hess\overline{u}(t,x)(V,V)$)
and $\mathscr{Y}_{k}[W]$ is the remainder term. By Taylor's theorem,
we have
\begin{equation}
|\mathscr{Y}_{k}[W]|\le C\|\overline{u}\|_{\mathcal{C}^{3}(\mathbb{R}^{d})}\left(\eps^{4}|\sigma(k+1)-\sigma(k)|^{2}+\eps^{3}|W_{\sigma(k+1)}-W_{\sigma(k)}|^{3}\right).\label{eq:Ejbound}
\end{equation}
Note that the second term of the second line of \eqref{ubartaylor}
appears in the definition \eqref{Itildedef} of $\widetilde{\mathscr{I}}_{t}^{\eps}$.
Thus, we can telescope the left side of \eqref{ubartaylor} to obtain
\begin{equation}
\widetilde{\mathscr{I}}_{t}^{\eps}[W]=u_{0}(\eps W_{\eps^{-2}t})-\overline{u}(t-\eps^{2}\sigma(0),\eps W_{\sigma(0)})+\sum_{k=0}^{K_{t}^{\eps}}\left(\eps^{2}\mathscr{X}_{k}[W]+\mathscr{Y}_{k}[W]\right),\label{eq:Itildeexpr}
\end{equation}
where
\begin{multline*}
\mathscr{X}_{k}[W]=(\sigma(k+1)-\sigma(k))\partial_{t}\overline{u}(t-\eps^{2}\sigma(k),\eps W_{\sigma(k)})-\frac{1}{2}\mathrm{Q}\overline{u}(t-\eps^{2}\sigma(k),\eps W_{\sigma(k)})(W_{\sigma(k+1)}-W_{\sigma(k)})\\
=(\sigma(k+1)-\sigma(k))\frac{1}{2}a\Delta\overline{u}(t-\eps^{2}\sigma(k),\eps W_{\sigma(k)})-\frac{1}{2}\mathrm{Q}\overline{u}(t-\eps^{2}\sigma(k),\eps W_{\sigma(k)})(W_{\sigma(k+1)}-W_{\sigma(k)}).
\end{multline*}
We now deal with each piece of this expression in term.

\emph{The drift terms}. We first define
\[
\widetilde{\mathscr{X}}_{k}=(\tilde{\sigma}(k+1)-\tilde{\sigma}(k))\frac{1}{2}a\Delta\overline{u}(t-\eps^{2}\sigma(k),\eps W_{\sigma(k)})-\frac{1}{2}\mathrm{Q}\overline{u}(t-\eps^{2}\sigma(k),\eps W_{\sigma(k)})(W_{\tilde{\sigma}(k+1)}-W_{\tilde{\sigma}(k)}),
\]
where $\tilde{\sigma}(k)=\min\{r\ge r_{k}\mid\eta_{r}^{W}=1\}$ differs
from $\sigma(k)$ by not being restricted to be less than $\eps^{-2}t$.
Using the relation \eqref{adef} between the effective diffusivity
$a$ and the variance of the increments $W_{\sigma_{n+1}^{W}}-W_{\sigma_{n}^{W}}$,
as well as the isotropy of $W$, we see that
\begin{equation}
\widetilde{\mathbb{E}}_{W}\widetilde{\mathscr{X}}_{k}[W]=0\label{eq:Xtildekexp0}
\end{equation}
for each $k$. We also note the simple bound
\[
|\widetilde{\mathscr{X}}_{k}[W]|\le a\|\overline{u}\|_{\mathcal{C}^{2}(\mathbb{R}^{d})}(\eps^{-\gamma}+Y)+\|\overline{u}\|_{\mathcal{C}^{2}(\mathbb{R}^{d})}(F(\eps^{-\gamma}+Y))^{2}\le a\|\overline{u}\|_{\mathcal{C}^{2}(\mathbb{R}^{d})}(\eps^{-\gamma}+Y)+\|\overline{u}\|_{\mathcal{C}^{2}(\mathbb{R}^{d})}\eps^{-\gamma}Z^{2}.
\]
Therefore, by \lemref{errorterms}, we have
\begin{equation}
\widetilde{\mathbb{E}}_{W}|\widetilde{\mathscr{X}}_{k}[W]|^{p}\le C\eps^{-p\xi}\label{eq:Xkmoment}
\end{equation}
for any $\xi>\gamma$. We further define
\[
M_{\ell}=\sum_{k=0}^{\ell}\widetilde{\mathscr{X}}_{k}[W].
\]
For each $\ell\ge0$, define $\mathcal{G}_{\ell}$ to be the $\sigma$-algebra
generated by $\{W_{t}\mid t\le\sigma(\ell)\}\cup\{\eta_{t}^{W}\mid t\le\sigma(\ell)\}$.
Then, according to \eqref{Xtildekexp0}, $\{M_{\ell}\}$ is
a martingale with respect to the filtration $\{\mathcal{G}_{\ell}\}$.
An $L^{p}$-version of the Burkholder--Gundy inequality as in \cite{DFJ68}
(see also \cite[Theorem 9]{Bur66}) implies that
\begin{equation}
\left(\widetilde{\mathbb{E}}_{W}|\eps^{2}M_{K_{t}^{\eps}}|^{p}\right)^{1/p}\le C\eps^{2}\left[(K_{t}^{\eps}+1)^{p/2-1}\sum_{k=0}^{K_{t}^{\eps}}\widetilde{\mathbb{E}}_{W}|\widetilde{\mathscr{X}}_{k}[W]|^{p}\right]^{1/p}\le C\eps^{\zeta}\label{eq:drifttermbound}
\end{equation}
for any $\zeta<1-\gamma/2$, where in the second inequality we used
\eqref{Kdef} and \eqref{Xkmoment}.

On the other hand, we note that $\mathscr{X}_{k}[W]-\widetilde{\mathscr{X}}_{k}[W]$
can be nonzero for at most one $k$, so we have
\[
\left|\sum_{k=0}^{K_{t}^{\eps}}(\mathscr{X}_{k}[W]-\widetilde{\mathscr{X}}_{k}[W])\right|=\max_{k=0}^{K_{t}^{\eps}}\left|\mathscr{X}_{k}[W]-\widetilde{\mathscr{X}}_{k}[W]\right|\le C\|\overline{u}\|_{\mathcal{C}^{2}(\mathbb{R}^{d})}\eps^{-\gamma/2}Z,
\]
so
\begin{equation}
\left(\widetilde{\mathbb{E}}_{W}\left|\eps^{2}\sum_{k=0}^{K_{t}^{\eps}}(\mathscr{X}_{k}[W]-\widetilde{\mathscr{X}}_{k}[W])\right|^{p}\right)^{1/p}\le C\eps^{\zeta}\label{eq:XkXktildebound}
\end{equation}
for any $\zeta<2-\gamma/2$.

\emph{The error term}. By \eqref{Ejbound}, we have a constant $C$ so that
\begin{multline*}
\left|\sum_{k=0}^{K_{t}^{\eps}}\mathscr{Y}_{j}^{\eps}[W]\right|  \le C\|\overline{u}\|_{\mathcal{C}^{3}(\mathbb{R}^d)}\sum_{j\ge0}^{K_{t}^{\eps}}\left(\eps^{4}|\sigma(k+1)-\sigma(j)|^{2}+|\eps W_{\sigma(k+1)}-\eps W_{\sigma(k)}|^{3}\right)\\
  \le C\|\overline{u}\|_{\mathcal{C}^{3}(\mathbb{R}^d)}K_{t}^{\eps}\left(\eps^{4}(\eps^{-\gamma}+Y)^{2}+|\eps F(\eps^{-\gamma}+Y)|^{3}\right)\le C\|\overline{u}\|_{\mathcal{C}^{3}(\mathbb{R}^d)}\left(\eps^{2-\gamma}(1+\eps^{\gamma}Y)^{2}+\eps^{1-\gamma/2}Z^{3}\right),
\end{multline*}
so by \lemref{errorterms} we have
\begin{equation}
\left(\widetilde{\mathbb{E}}_{W}\left|\sum_{k=0}^{K_{t}^{\eps}}\mathscr{Y}_{j}^{\eps}[W]\right|^{p}\right)^{1/p}\le C\eps^{\zeta}\label{eq:errortermbound}
\end{equation}
for any $\zeta<1-\gamma/2$.%

\emph{The initial term.} Finally, we observe that
\begin{equation}
\left(\widetilde{\mathbb{E}}_{W}\left|\overline{u}(t-\eps^{2}\sigma(0),\eps W_{\sigma(0)})-\overline{u}(t,x)\right|^{p}\right)^{1/p}\le\|\overline{u}\|_{\mathcal{C}^{1}(\mathbb{R}^d)}\left(\widetilde{\mathbb{E}}_{W}(\eps^{2}\sigma(0)+\eps|W_{\sigma(0)}|)^{p}\right)^{1/p}\le C\eps^{\zeta}\label{eq:initialtermbound}
\end{equation}
for any $\zeta<1-\gamma/2$.

Applying the bounds \eqref{drifttermbound}, \eqref{XkXktildebound},
\eqref{errortermbound}, and \eqref{initialtermbound} to \eqref{Itildeexpr}
gives us \eqref{boundwithItilde}.
\end{proof}
\begin{cor}
\label{cor:triangleineq}For any $1\le p<\infty$ and $\zeta<(\gamma-1)\wedge(1-\gamma/2)$
there exists $C=C(p,t,\zeta,\|u_{0}\|_{\mathcal{C}^{3}(\mathbb{R}^d)})$ so that
\[
\left(\widetilde{\mathbb{E}}_{W}^{\eps^{-1}x}|u_{0}(\eps W_{\eps^{-2}t})-\overline{u}(t,x)-\mathscr{I}_{t}^{\eps}[W]|^{p}\right)^{1/p}\le C\eps^{\zeta}.
\]
\end{cor}

\begin{proof}
This is a simple consequence of the $L^{p}$ triangle inequality applied
to the results of the last two lemmas.
\end{proof}
We will also need the following auxiliary lemma.
\begin{lem}
\label{lem:intersectionmoment}There is a $\beta_{0}>0$ so that if
$\chi>1$, $\beta>0$ are such that $\chi\beta^{2}<\beta_{0}^{2}$,
then there is a constant $C=C(\chi,\beta)<\infty$ so that for any
$\eps>0$ and $x,\tilde{x}\in\mathbb{R}^{2}$ we have
\[
\widetilde{\mathbb{E}}_{W,\widetilde{W}}^{\eps^{-1}x,\eps^{-1}\tilde{x}}\left(\exp\left\{ \beta^{2}\mathscr{R}_{\eps^{-2}t}[W,\widetilde{W}]\right\} -1\right)^{\chi}\le C\left(\frac{\eps}{|x-\tilde{x}|}\wedge1\right)^{d-2}.
\]
\end{lem}

\begin{proof}
Since $\mathscr{R}_{t}[W,\widetilde{W}]\ge0$, we have
\begin{align*}
&\widetilde{\mathbb{E}}_{W,\widetilde{W}}^{\eps^{-1}x,\eps^{-1}\tilde{x}}\left(\exp\left\{ \beta^{2}\mathscr{R}_{\eps^{-2}t}[W,\widetilde{W}]\right\} -1\right)^{\chi}\\
&\qquad\le\widetilde{\mathbb{P}}_{W,\widetilde{W}}^{\eps^{-1}x,\eps^{-1}\tilde{x}}\left(\mathscr{R}_{\eps^{-2}t}[W,\widetilde{W}]>0\right)\\&\qquad\qquad\qquad\times\sup_{r>0,W|_{[0,r]},\widetilde{W}|_{[0,r]}}\widetilde{\mathbb{E}}_{W,\widetilde{W}}^{\eps^{-1}x,\eps^{-1}\tilde{x}}\left[\exp\left\{ \chi\beta^{2}\mathscr{R}_{[r,\eps^{-2}t]}[W,\widetilde{W}]\right\}\;\middle|\;W|_{[0,r]},\widetilde{W}|_{[0,r]}\right] \\
&\qquad\le C\left(\frac{\eps}{|x-\tilde{x}|}\wedge1\right)^{d-2}
\end{align*}
by \propref{hittingprob} and \propref{expexpbdd}, as long as $\chi\beta^{2}$
is sufficiently small.
\end{proof}
\begin{prop}
\label{prop:uniformconv}For all $\chi>1$, $\zeta<(1-\gamma/2)\wedge(\gamma-1)$, and $t>0$,
there exists a constant $C=C(\chi,\zeta,t,\|u\|_{\mathcal{C}^3(\mathbb{R}^d)})$ so that
\[
\left|\mathbf{E}q^{\eps}(t,x)q^{\eps}(t,\tilde{x})-\mathbf{E}q^{\eps}(t,x)\mathbf{E}q^{\eps}(t,\tilde{x})\right|\le C|x-\tilde{x}|^{-\frac{d-2}{\chi}}\eps^{2\zeta+\frac{d-2}{\chi}}.%
\]
\end{prop}

\begin{proof}
Take $p\ge1$ so that $1/\chi+2/p=1$. We go back to \eqref{qcovcomp}
and apply Hölder's inequality, as well as \corref{triangleineq} and
\lemref{intersectionmoment}, to get the bound
\begin{align*}
\left|\mathbf{E}q^{\eps}(t,x)q^{\eps}(t,\tilde{x})-\mathbf{E}q^{\eps}(t,x)\mathbf{E}q^{\eps}(t,\tilde{x})\right| & \le\|\mathscr{G}\|_{\infty}^{2}\left(\sup_{x\in\mathbb{R}^{d}}\widetilde{\mathbb{E}}_{W}^{\eps^{-1}x}|u_{0}(\eps W_{\eps^{-2}t})-\overline{u}(t,x)-\mathscr{I}_{t}^{\eps}[W]|^{p}\right)^{2/p}\\
 & \qquad\times\left(\widetilde{\mathbb{E}}_{W,\widetilde{W}}^{\eps^{-1}x,\eps^{-1}\tilde{x}}\left(\exp\left\{ \beta^{2}\mathscr{R}_{\eps^{-2}t}[W,\widetilde{W}]\right\} -1\right)^{\chi}\right)^{1/\chi}\\
 & \le C\eps^{2\zeta}\left(\frac{\eps}{|x-\tilde{x}|}\right)^{\frac{d-2}{\chi}}.\qedhere
\end{align*}
\end{proof}
We are finally ready to prove \thmref{weakconvergence}.
\begin{proof}[Proof of \thmref{weakconvergence}.]
By \propref{uniformconv}, we have, for any $\zeta<(1-\gamma/2)\wedge(\gamma-1)$
and any $\chi>1$, that
\begin{align*}
\eps^{-(d-2)} & \mathbf{E}\left(\int g(x)q^{\eps}(t,x)\,\dif x-\mathbf{E}\int g(x)q^{\eps}(t,x)\,\dif x\right)^{2}\\
 & =\eps^{-(d-2)}\int\int g(x)g(\tilde{x})\left[\mathbf{E}q^{\eps}(t,x)q^{\eps}(t,\tilde{x})-\mathbf{E}q^{\eps}(t,x)\mathbf{E}q^{\eps}(t,\tilde{x})\right]\,\dif x\,\dif\tilde{x}\\
 & \le C\eps^{(d-2)(1/\chi-1)+2\zeta}\int\int g(x)g(\tilde{x})|x-\tilde{x}|^{-\frac{d-2}{\chi}}\,\dif x\,\dif\tilde{x}.
\end{align*}
The integral in the last line is finite because $g$ is smooth and
compactly supported. Now by taking $\chi$ sufficiently close to $1$ and reducing $\zeta$ slightly, we achieve \eqref{weakfluctsbound}.
\end{proof}

\bibliographystyle{habbrv}
\bibliography{citations}

\end{document}